\documentclass{article}

\usepackage{comment}
\usepackage{amsfonts}
\usepackage{amssymb}
\usepackage{amsthm}
\usepackage{amsmath}

\usepackage{cite}
\usepackage{amsrefs}

\usepackage{longtable}

\usepackage[all,cmtip]{xy}

\usepackage{fancyhdr}

\usepackage{graphicx}
\usepackage{hyperref}

\usepackage{tikz-cd}

\usepackage{aurical}

\usepackage{stmaryrd}

\usepackage[T1]{fontenc}

\usepackage[outline]{contour}
\usepackage{xcolor}

\usepackage{array}

\newcommand{\what}[1]{\widehat{#1}}

\newcommand{\et}{{\acute{e}t}}

\newcommand{\CC}{\mathbb{C}}

\newcommand{\FF}{\mathbb{F}}

\newcommand{\NN}{\mathbb{N}}

\newcommand{\QQ}{\mathbb{Q}}

\newcommand{\ZZ}{\mathbb{Z}}

\newcommand{\Qbar}{\overline{\mathbb{Q}}}

\newcommand{\cC}{\mathcal{C}}

\newcommand{\eE}{\mathcal{E}}

\newcommand{\hH}{\mathcal{H}}

\newcommand{\mM}{\mathcal{M}}

\newcommand{\Zhat}{\widehat{\mathbb{Z}}}

\newcommand{\Ahat}{\widehat{A}}

\newcommand{\la}[1]{\,^{#1}\!} 

\newcommand{\mf}[1]{\mathfrak{#1}} 

\newcommand{\Aut}{\operatorname{Aut}}
\newcommand{\SAut}{\operatorname{SAut}}
\newcommand{\Inn}{\operatorname{Inn}}
\newcommand{\Out}{\operatorname{Out}}
\newcommand{\IOut}{\operatorname{IOut}}
\newcommand{\IAut}{\operatorname{IAut}}
\newcommand{\IAEnd}{\operatorname{IAEnd}}
\newcommand{\IA}{\operatorname{IA}}
\newcommand{\Ker}{\operatorname{Ker }}

\newcommand {\spmatrix}[4]{\left[\begin{smallmatrix}#1&#2\\#3&#4\end{smallmatrix}\right]}
\newcommand {\ttmatrix}[4]{\left[\begin{array}{rr}#1&#2\\#3&#4\end{array}\right]}

\newcommand{\cvector}[2]{\left[\begin{array}{c}#1\\#2\end{array}\right]}

\newcommand{\Gal}{\operatorname{Gal}}
\newcommand{\id}{\operatorname{id}}
\newcommand{\Hom}{\operatorname{Hom}}

\newcommand{\Epi}{\operatorname{Epi}}

\newcommand{\GL}{\operatorname{GL}}
\newcommand{\SL}{\operatorname{SL}}

\newcommand{\Stab}{\operatorname{Stab}}

\newcommand{\cotimes}{\,\widehat{\otimes}\,}

\newcommand{\ab}{\textrm{ab}}
\newcommand{\meta}{\textrm{meta}}

\newcommand{\Cyc}{\operatorname{Cyc}}

\newcommand{\rightiso}{\stackrel{\sim}{\longrightarrow}}

\newcommand{\ol}[1]{\overline{#1}}

\newcommand{\Spec}{\operatorname{Spec }}

\newcommand{\End}{\operatorname{End}}

\newcommand{\ext}{\text{ext}}

\newcommand{\ps}[1]{[\![#1]\!]}

\newcommand{\tp}{\text{top}}

\newcommand{\CF}{\operatorname{CF}}

\newcommand{\PMod}{\underline{\textbf{PMod}}}
\newcommand{\PAb}{\underline{\textbf{PAb}}}
\newcommand{\EExt}{\underline{\textbf{Ext}}}
\newcommand{\Sets}{\underline{\textbf{Sets}}}

\newcommand{\univ}{\text{univ}}

\newcommand{\sgap}{\vspace{0.2cm}}

\theoremstyle{definition}\newtheorem{defn}{Definition}[subsection]
\theoremstyle{remark}
\theoremstyle{remark}
\theoremstyle{remark}
\theoremstyle{remark}\newtheorem{remark}[defn]{Remark}
\theoremstyle{remark}\newtheorem*{remark*}{Remark}
\theoremstyle{remark}
\theoremstyle{remark}
\theoremstyle{remark}\newtheorem{example}[defn]{Example}

\theoremstyle{plain}\newtheorem{prop}[defn]{Proposition}
\theoremstyle{plain}\newtheorem{thm}[defn]{Theorem}
\theoremstyle{plain}\newtheorem{open}[defn]{Open Problem}
\theoremstyle{plain}\newtheorem*{thm*}{Theorem}
\theoremstyle{plain}\newtheorem{lemma}[defn]{Lemma}
\theoremstyle{plain}\newtheorem{cor}[defn]{Corollary}
\theoremstyle{plain}
\theoremstyle{plain}\newtheorem*{conj*}{Conjecture}
\theoremstyle{plain}\newtheorem*{prop*}{Proposition}

\renewcommand{\det}{\operatorname{det}}

\title{Arithmetic monodromy actions on pro-metabelian fundamental groups of once-punctured elliptic curves}
\author{William Yun Chen, Pierre Deligne}

\begin{document}
\maketitle

\tableofcontents

\section{Introduction}\label{section_introduction}
\subsection{The Problem}\label{ss_problem}
Let $F_2$ denote the free group on 2 generators. The abelianization $\ab : F_2\rightarrow \ZZ^2$ induces a surjective homomorphism $\ab_* : \Aut(F_2)\rightarrow\GL_2(\ZZ)$. A classical theorem of Nielsen (c.f. \cite{MKS04} \S3.5) implies that the kernel of $\ab_*$ is the subgroup $\Inn(F_2)$ of inner automorphsims of $F_2$. In other words, $\ab_*$ factors through an isomorphism $\Out(F_2) := \Aut(F_2)/\Inn(F_2)\rightiso \Aut(\ZZ^2)\cong\GL_2(\ZZ)$. 

\sgap

Given a finite 2-generated group $G$, let $\Epi^\ext(F_2,G)$ be the set of $G$-conjugacy classes of surjective homomorphisms from $F_2$ to $G$. There is a natural action of $\Out(F_2)$ on this set, and hence a natural action of $\SL_2(\ZZ)\subset\GL_2(\ZZ) = \Out(F_2)$. Given $\varphi : F_2\twoheadrightarrow G$, let $\Gamma_\varphi$ denote the stabilizer in $\SL_2(\ZZ)$ of the image of $\varphi$ in $\Epi^\ext(F_2,G)$.

\sgap

A finite index subgroup $\Gamma\le\SL_2(\ZZ)$ is \emph{congruence} if it contains $\Gamma(n) := \Ker(\SL_2(\ZZ)\rightarrow\SL_2(\ZZ/n))$ for some $n$. The purpose of this paper is to investigate special cases of the

\begin{open}\label{zee_probrem} For which $\varphi : F_2\twoheadrightarrow G$ is $\Gamma_\varphi$ a congruence subgroup of $\SL_2(\ZZ)$? Can we describe the orbits of $\SL_2(\ZZ)$ on $\Epi^\ext(F_2,G)$?
\end{open}

\subsection{Motivation}\label{ss_motivation}
This combinatorial problem \ref{zee_probrem} has a geometric origin. In \cite{Chen17}, the first author studies the notion of $G$-structures on elliptic curves, where $G$ is a finite 2-generated group.  Let $E$ be an elliptic curve with origin $O$, and let $E^\circ$ be the complement of $\{O\}$. If $E$ is over an algebraically closed field $\ol{K}$, a $G$-structure on $E$ is an element of $\Epi^\ext(\pi_1^\et(E^\circ,x),G)$ for some choice of base point $x$. If $x$ and $y$ are two base points, there is a canonical conjugacy class of isomorphisms between $\pi_1^\et(E^\circ,x)$ and $\pi_1^\et(E^\circ,y)$. Because of this, the choice of base point will be immaterial for our purposes, and in the following base points will often be omitted.

\sgap

Equivalently, a $G$-structure is an isomorphism class of connected $G$-torsors on $E^\circ$. Indeed, to give a $G$-torsor $P$ on $E^\circ$ (with $G$ acting on the right) amounts to giving its fiber $P_x$ at $x$, and a left action $\sigma$ of $\pi_1^\et(E^\circ,x)$ on the $G$-torsor $P_x$. The choice of a $p\in P_x$ defines an isomorphism of $G$-torsors $G\rightiso P_x$ sending $g\mapsto pg$ for any $g\in G$. This choice of $p$ hence identifies $\sigma$ with the homomorphism $\varphi : \pi_1^\et(E^\circ,x)\rightarrow G$ defined by $\sigma(\gamma)(pg) = p\varphi(\gamma)g$ for $\gamma\in\pi_1^\et(E^\circ,x)$. This homomorphism $\varphi$ is surjective if and only if $P$ is connected. If $p$ is replaced by $ph$ for $h\in G$, then $\varphi$ is changed into the homomorphism $\gamma\mapsto h^{-1}\varphi(\gamma)h$.

\sgap

The relation with \S\ref{ss_problem} is that if $\ol{K} = \CC$, then $\pi_1^\tp(E^\circ(\CC))$ is a free group on two generators, and $\pi_1^\et(E^\circ)$ can be identified with its profinite completion, and hence is a free profinite group of rank 2.

\sgap

For $E$ an elliptic curve over a general base scheme $S$, one should consider the presheaf which to any $U\rightarrow S$ attaches the set of isomorphism classes of geometrically connected $G$-torsors $P$ on $E^\circ_U := E^\circ\times_S U$. Here, ``geometrically connected'' means that the geometric fibers of $P\rightarrow U$ are connected. A $G$-structure on $E$ is then a global section of the associated \'{e}tale sheaf. If $S$ is the spectrum of a field $K$ with algebraic closure $\ol{K}$, a $G$-structure on $E$ is a $G$-structure on $E_{\ol{K}}$ which is fixed by $\Gamma_K := \Gal(\ol{K}/K)$. If $G$ is centerless, geometrically connected $G$-torsors have no nontrivial automorphisms, and hence a $G$-structure is simply the data of a geometrically connected $G$-torsor up to (unique) isomorphism. If $G$ is abelian and $|G|$ prime to all residue characteristics, a $G$-torsor on $E^\circ$ extends to $E$, and $G$-structures can be identified with geometrically connected $G$-torsors on $E$, trivialized above the origin $O$. In this case they are equivalent to classical level structures (c.f. \cite{Chen17} 2.2.12).

\sgap

In \cite{Chen17} it is shown that the moduli of elliptic curves equipped with $G$-structures gives rise to moduli stacks $\mM(G)$ over $\QQ$, such that the forgetful functor gives a finite \'{e}tale morphism $\mM(G)\stackrel{p}{\rightarrow}\mM(1)$, where $\mM(1)$ is the moduli stack of elliptic curves over $\QQ$. More generally, this holds over $\ZZ[1/|G|]$. The complex analytic stack (orbifold) $\mM(1)(\CC)$ is $[\hH/\SL_2(\ZZ)]$. If $E$ is an elliptic curve over $\CC$ with homology $H_1E := H_1(E(\CC),\ZZ)$ and $x_E$ is the corresponding base point of $\mM(1)(\CC)$, the action of $\pi_1^\tp(\mM(1)(\CC),x_E)$ on $H_1E$ is given by an isomorphism $\pi_1^\tp(\mM(1)(\CC),x_E)\rightiso\SL(H_1E)$ which we will use to identify the two groups. The \'{e}tale fundamental group is the profinite completion $\what{\SL(H_1E)}$.

\sgap

The fiber $p^{-1}(x_E)$ of the finite \'{e}tale map $\mM(G)_\CC\stackrel{p}{\rightarrow}\mM(1)_\CC$ at $x_E$ can be canonically identified with the set
$$\Epi^\ext(\pi_1^\et(E^\circ),G) \cong \Epi^\ext(\pi_1^\tp(E^\circ(\CC)),G)\cong \Epi^\ext(F_2,G)$$
The abelianization of the free group $\pi_1^\tp(E^\circ(\CC))$ is $H_1(E^\circ(\CC),\ZZ)\rightiso H_1E$, and the action of $\pi_1(\mM(1)(\CC),x_E) = \SL(H_1E)$ on $p^{-1}(x_E)$ is given as in \S\ref{ss_problem}: $\SL(H_1E)$ is viewed as an index 2 subgroup of $\Out(\pi_1^\tp(E^\circ(\CC))\cong\GL(H_1E)$, which naturally acts on $\Epi^\ext(\pi_1^\tp(E^\circ(\CC),G)$. By the Galois correspondence, we get
$$\mM(G)_\CC = \bigsqcup_{\varphi} [\hH/\Gamma_\varphi]$$
where $\varphi$ ranges over representatives of $\SL_2(\ZZ)$-orbits on $\Epi^\ext(F_2,G)$, and $\Gamma_\varphi := \Stab_{\SL_2(\ZZ)}(\varphi)$.

\sgap

By a theorem of Asada \cite{Asa01}, one finds that \emph{every} finite index subgroup $\Gamma\le\SL_2(\ZZ)$ contains $\Gamma_\varphi$ for a suitable group $G$ and surjection $\varphi : F_2\twoheadrightarrow G$. Thus, as a result, we can interpret every modular curve - that is, $\hH/\Gamma$ for some finite index subgroup $\Gamma\le\SL_2(\ZZ)$ - as a quotient of a moduli space of elliptic curves equipped with a suitable nonabelian $G$-structure. However, Asada's proof says very little about the relation between $\varphi$, and $\Gamma_\varphi$.

\sgap

We will say that a finite 2-generated group $G$ is \emph{congruence} if all stabilizers $\Gamma_\varphi$ of all surjections $\varphi : F_2\twoheadrightarrow G$ are congruence subgroups of $\SL_2(\ZZ)$, or what is the same, if all components of $\mM(G)_\CC$ are congruence modular curves, and noncongruence if this does not hold. We'll say that $G$ is \emph{purely noncongruence} if all components of $\mM(G)_\CC$ are noncongruence.

\sgap

While it is clear that all abelian groups $G$ are congruence, the converse is not true, and there are many examples of surjections $\varphi$ onto nonabelian groups with $\Gamma_\varphi$ nonetheless congruence. In \cite{Chen17} \S4.2, the first author shows that the dihedral groups $D_{2k}$ are congruence, and in \S4.4, gives an example of an infinite family of nonabelian finite simple groups which are noncongruence. In general, this seems to be rather mysterious, but the expectation is that any sufficiently nonabelian group - in particular, any nonabelian finite simple group - should be purely noncongruence.

\subsection{Overview of main results}\label{ss_overview}

The purpose of this paper is to study the case when $G$ is a finite 2-generated metabelian group\footnote{a metabelian group is a group whose commutator subgroup is abelian.}. In some sense this is as close to being abelian as we can get, without actually being abelian. The basic result is that for any such $G$ and any surjection $\varphi : F_2\twoheadrightarrow G$, $\Gamma_\varphi := \Stab_{\SL_2(\ZZ)}(\varphi)$ is a congruence subgroup, and moreover, $\mM(G)_\QQ$ is a disjoint union of copies of moduli stacks of elliptic curves with classical level structures. Since any map from $F_2$ to a finite metabelian group $G$ factors through the rank 2 free profinite metabelian group $\what{M}$, we are led to study $\what{M}$ and its automorphisms. The homogeneous disjoint union decomposition of $\mM(G)_\QQ$ is a consequence of a surprising group-theoretic result that all endomorphisms of $\what{M}$ which induce the identity on the abelianization leaves every open normal subgroup stable.

\sgap

For simplicity we have stated our results over subfields $K\subset\Qbar$. However, since our methods are almost entirely group-theoretic, analogous results hold over arbitrary $\ZZ[1/|G|]$-schemes. Moreover, our results on $\Out(\what{M})$ may be of interest to group-theorists. The reader only interested in the group theory can skip to \S\ref{sss_group_theory}.

\subsubsection{Geometric preliminaries}\label{sss_geometric_preliminaries}

Let $\Qbar$ denote the algebraic closure of $\QQ$ inside $\CC$. Let $K$ be a subfield of $\Qbar$, and $\ol{K} := \Qbar$. Let $E$ be an elliptic curve over $K$, corresponding to a $K$-point $x : \Spec K\rightarrow\mM(1)_K$. Similarly, $E_{\ol{K}}$ determines a geometric point $\ol{x} : \Spec\ol{K}\rightarrow\mM(1)_K$. Then there is a homotopy exact sequence of fundamental groups
$$1\rightarrow\pi_1^\et(\mM(1)_{\Qbar},\ol{x})\rightarrow\pi_1^\et(\mM(1)_K,\ol{x})\rightarrow\Gamma_K\rightarrow 1$$
which is split by the $K$-point $x$. Again letting $H_1E := H_1(E(\CC),\ZZ)$, we have $\pi_1^\et(\mM(1)_{\Qbar},\ol{x}) = \what{\SL(H_1E)}\cong \what{\SL_2(\ZZ)}$, so we may identify $\pi_1^\et(\mM(1)_K,\ol{x}) = \what{\SL(H_1E)}\rtimes\Gamma_K\cong \what{\SL_2(\ZZ)}\rtimes\Gamma_K$. Let $\eE^\circ$ be the universal punctured elliptic curve over $\mM(1)_K$, then $E^\circ$ is the fiber of $\eE^\circ$ at $x$, and for any geometric point $\ol{y}$ on $E^\circ$, we have another exact sequence
$$1\rightarrow\pi_1^\et(E_{\ol{K}}^\circ,\ol{y})\rightarrow\pi_1^\et(\eE^\circ,\ol{y})\rightarrow\pi_1^\et(\mM(1)_K,\ol{x})\rightarrow 1$$
from which we get a monodromy representation
$$\rho_{E^\circ/K} : \pi_1^\et(\mM(1)_K,\ol{x})=\what{\SL(H_1E)}\rtimes\Gamma_K \rightarrow \Out(\pi_1^\et(E^\circ_{\ol{K}},\ol{y}))\cong\Out(\what{F_2})$$
such that the restriction to $\what{\SL(H_1E)}$ is precisely the action described earlier.

\sgap

Let $\cC$ be a class of finite groups closed under quotients and subdirect products, and containing all finite abelian groups. Let $\pi_1^\et(E^\circ_{\ol{K}})^\cC$ be the maximal pro-$\cC$ quotient of $\pi_1^\et(E^\circ_{\ol{K}})$. Since it is a characteristic quotient, it inherits from $\rho_{E^\circ/K}$ an outer action:
$$\rho_{E^\circ/K}^\cC : \pi_1^\et(\mM(1)_K,\ol{x})\rightarrow \Out(\pi_1^\et(E^\circ_{\ol{K}})^\cC)$$
If $\cC = \ab$ is the class of all finite abelian groups, then the natural inclusion $i : E^\circ_{\ol{K}}\hookrightarrow E_{\ol{K}}$ induces a map $\pi_1^\et(E^\circ_{\ol{K}})\stackrel{i_*}{\rightarrow}\pi_1^\et(E_{\ol{K}})$ which group-theoretically is just abelianization, and hence we have $\pi_1^\et(E_{\ol{K}})\cong\pi_1^\et(E^\circ_{\ol{K}})^\ab$. The corresponding representation $\rho^\ab_{E^\circ/K} : \pi_1^\et(\mM(1)_K) =\what{\SL(H_1E)}\rtimes\Gamma_K\rightarrow \Out(\pi_1^\et(E^\circ_{\ol{K}})^\ab) = \GL(\what{H_1E})$ can be identified with the representation $\rho_{E/K}$ coming from the unpunctured $E$, which sends $\what{\SL(H_1E)}\cong\what{\SL_2(\ZZ)}$ onto $\SL_2(\Zhat)\cong\SL(\what{H_1E})\subset\GL(\what{H_1E})$, and sends $\Gamma_K$ to $\GL(\what{H_1E})$ via the standard Galois representation on the Tate module of $E$. In a diagram, we have

\[\begin{tikzcd}
& & & \Out(\pi_1^\et(E^\circ_{\ol{K}}))\arrow[d] \\
\what{\SL_2(\ZZ)}\rtimes\Gamma_K\cong \what{\SL(H_1E)}\rtimes\Gamma_K\arrow[r,equals] & \pi_1^\et(\mM(1)_K,\ol{x})\arrow[rru,bend left = 25,"\rho_{E^\circ/K}"]\arrow[rr,"\rho_{E^\circ/K}^\cC"]\arrow[rrd,bend right=25,"\rho_{E/K}"'] & & \Out(\pi_1^\et(E^\circ_{\ol{K}})^\cC)\arrow[d,"\ab_*"] \\
& & & \Out(\pi_1^\et(E_{\ol{K}}))\arrow[r,equals] & \GL(\what{H_1E}) \cong\GL_2(\Zhat)
\end{tikzcd}\]

\begin{defn} We will say that a class $\cC$ as above is \emph{geometrically congruence} if $\rho_{E^\circ/K}^\cC|_{\what{\SL(H_1E)}}$ factors through the congruence quotient $\SL_2(\Zhat)\cong\SL(\what{H_1E})$, and we will say that $\cC$ is \emph{arithmetically congruence} if the image of $\rho_{E^\circ/K}^\cC$ is mapped isomorphically onto the image of $\rho_{E^\circ/K}^\ab = \rho_{E/K}$ for every subfield $K\subset\Qbar$.
\end{defn}

Note that the two notions coincide when $K = \Qbar$, and hence being arithmetically congruence implies being geometrically congruence. Furthermore, since all representations above are continuous, if $\rho_{E^\circ/K}|_{\what{\SL(H_1E)}}$ factors through the congruence quotient $\SL(\what{H_1E})$, then for any $G\in\cC$ and any surjection $\varphi : F_2\cong \pi_1^\tp(E^\circ(\CC))\twoheadrightarrow G$, $\Gamma_\varphi$ will correspond to an open subgroup of $\SL(\what{H_1E})\cong\SL_2(\Zhat)$, and hence will be congruence.

\sgap

We can finally state our main results.

\begin{thm*}[\ref{thm_metabelian_congruence}, \ref{cor_M1_monodromy}] If $\cC = \meta$ is the class of metabelian groups, then $\cC$ is arithmetically congruence.
\end{thm*}

Moreover, we find that the connected components of $\mM(G)_\QQ$ are all isomorphic:

\begin{thm*}[\ref{thm_metabelian_is_e_congruence}] Let $G$ be a 2-generated metabelian group, and let $\mM(G)_\QQ$ be the moduli stack of elliptic curves with $G$-structures. The group $\Aut(G)$ acts (via its quotient $\Out(G)$) as automorphisms of the cover $\mM(G)_\QQ\rightarrow\mM(1)_\QQ$, and we have
\begin{itemize}
\item $\Out(G)$ permutes transitively the connected components of $\mM(G)_\QQ$, which are hence all isomorphic, and
\item If $G$ is of exponent $e$, then there is a subgroup $H\le\GL_2(\ZZ/e)$ such that each connected component of $\mM(G)_\QQ$ is isomorphic to $\mM((\ZZ/e)^2)_\QQ/H$.
\end{itemize}
In particular, for any elliptic curve $E$ over a subfield $K\subset\Qbar$, $G$-structures on $E$ are equivalent to congruence structures of level $H$.
\end{thm*}

\subsubsection{Group theory}\label{sss_group_theory}
We now describe our methods. One should imagine that we have made the following identifications/definitions:
$$F := F_2 = \pi_1^\tp(E^\circ(\CC)),\qquad M := F/F'', \qquad A := F^\ab = H_1E := H_1(E(\CC),\ZZ)\;\;(\cong\ZZ^2)$$
$$\what{F} = \pi_1(E^\circ_{\ol{K}}),\qquad\what{M} = \pi_1(E^\circ_{\ol{K}})^\meta,\qquad \Ahat = \pi_1(E^\circ_{\ol{K}})^\ab = \pi_1(E_{\ol{K}})\;\;(\cong\Zhat^2)$$
Thus, $\what{M}$ is a free profinite metabelian group of rank 2 with abelianization $\what{A}$. Then, $\what{A}$ acts on the commutator subgroup $\what{M}'$ by conjugation, giving $\what{M}'$ the structure of a module under the completed group algebra $\Zhat\ps{\what{A}}$. Let $\what{T}\cong\Zhat\ps{\what{A}}^2$ be a free $\Zhat\ps{\what{A}}$ module of rank 2. Using a result of Remeslennikov \cite{Rem79}, we obtain an embedding 
$$\mu : \what{M}\hookrightarrow\what{T}\rtimes\what{A} \quad (\text{c.f. \S\ref{ss_profinite_magnus}})$$
respecting the projections to $\what{A}$. We show that this embedding is universal amongst embeddings of $\what{M}$ into split extensions of $\what{A}$. Let IA denote ``(I)dentity on (A)belianization'', then IA-endomorphisms of $\what{M}$ are precisely the endomorphisms of $\what{M}$ as an extension of $\what{A}$. By universality of the embedding $\mu$, IA-endomorphisms of $\what{M}$ naturally extend to $\Zhat\ps{\what{A}}$-linear endomorphisms of $\what{T}$, through which we have the canonical notion of the ``(Bachmuth) determinant'' (c.f. \S\ref{ss_bachmuth}) of an IA-endomorphism.

\sgap

If $x_1,x_2$ generate $\what{M}$, with images $a_1,a_2\in\what{A}$, then the image via $\mu$ of their commutator is
$$\mu([x_1,x_2]) = \mu(x_1x_2x_1^{-1}x_2^{-1}) = ((1-a_2)t_1 + (a_1-1)t_2,1)\in\what{T}\rtimes\what{A}$$
where $t_1,t_2$ is a basis of $\what{T}$ satisfying $\mu(x_i) = (t_i,a_i)$. By giving an explicit decomposition of the ring $\Zhat\ps{\what{A}}$ as a direct product of $p$-adic power series rings, we find that the elements $a_i-1\in\Zhat\ps{\what{A}}$ are not zero-divisors, and hence $\what{M}'$ is a \emph{free} $\Zhat\ps{\what{A}}$-module of rank 1. Using an explicit parametrization of all IA-endomorphisms (\S\ref{ss_IAEnds}), we furthemore find that $\mu(\what{M}')$ sits inside $\what{T} = \what{T}\rtimes 1$ as a common ``eigenspace'' of every IA-endomorphism, with eigenvalues given by the (Bachmuth) determinants. A key question is understanding the image of the determinant map, and to characterize the inner automorphisms. Let $\Zhat\ps{\what{A}}^{\times'}$ be the subgroup of $\Zhat\ps{\what{A}}^\times$ consisting of units with augmentation 1 (c.f. \ref{def_distinguished}). Then we have

\begin{thm*}[\ref{prop_det}, \ref{prop_IA_is_inner}] Given $\gamma\in\IAEnd(\what{M})$, $\gamma$ is an automorphism if and only if $\det(\gamma)\in\Zhat\ps{\what{A}}^{\times'}$, and $\gamma$ is an inner automorphism if and only if $\det(\gamma)\in\what{A}\subset\Zhat\ps{\what{A}}$.
\end{thm*}

Thus, we find that $\Out(\what{M})$ is an extension of $\GL(\Ahat)$ by the rather large abelian group $\Zhat\ps{\what{A}}^{\times'}/\what{A}$. By extending $\det : \IAut(\what{M})\rightarrow\Zhat\ps{\what{A}}^\times$ to a continuous crossed homomorphism $\det : \Aut(\what{M})\rightarrow\Zhat\ps{\what{A}}^\times$, we show that the image of $\SL(A)\hookrightarrow\Out(\what{M})$ has determinants in $\what{A}$. By continuity, the same is true of (the image of) $\what{\SL(A)}$, and the same is true of its congruence kernel $\bigcap_{n\ge 1}\ol{\Gamma(n)}$, which consists of (outer) IA-automorphisms, which by the above theorem implies that the image of the congruence kernel in $\Out(\what{M})$ is trivial. This proves

\begin{thm*}[\ref{thm_metabelian_congruence}] The image of $\what{\SL(A)} \cong \what{\SL_2(\ZZ)}$ in $\Out(\what{M})$ is $\SL(\Ahat)\cong\SL_2(\Zhat)$ (that is, $\cC = \meta$ is geometrically congruence) \footnote{After writing up our proof, the first author became aware of a result of Ben-Ezra \cite{BenEzra16} from 2016 from which one may also deduce this result.}
\end{thm*}

\sgap

Next, using the fact that the Weil pairing and the ``branch cycle argument'' describes the Galois action on both $\pi_1^\et(E^\circ_{\ol{K}})^\meta\cong\what{M}$ and its derived subgroup in terms of the cyclotomic character $\chi : \Gamma_K\rightarrow\Zhat^\times$, we find that the Galois image in $\Out(\what{M})$ is always contained in the subgroup $\Out^b(\what{M})$ of ``braid-like outer automorphisms'' (c.f. \ref{def_braid_like_automorphisms}), which maps isomorphically onto $\GL(\Ahat)$:

\begin{thm*}[\ref{thm_split}, \ref{cor_M1_monodromy}] The natural exact sequence induced by abelianization $\what{M}\rightarrow \Ahat$
$$1\rightarrow\Zhat\ps{\what{A}}^{\times'}/\what{A}\rightarrow\Out(\what{M})\stackrel{\ab_*}{\rightarrow}\GL(\Ahat)\rightarrow 1$$
is split by a section $\Out(\what{M})\stackrel{s}{\leftarrow}\GL(\Ahat)$ with image the subgroup $\Out^b(\what{M})$ of braid-like outer automorphisms. Furthermore, for any elliptic curve $E/K$, the image of $\Gamma_K$ under $\rho_{E^\circ/K}^\meta$ is contained in $\Out^b(\what{M})$. If $K = \QQ$, then the image $\rho_{E^\circ/K}^\meta(\pi_1^\et(\mM(1)_\QQ))$ is precisely $s(\GL(\Ahat)) = \Out^b(\what{M}) \cong \GL_2(\Zhat)$.
\end{thm*}


\sgap

This implies that the representation $\rho_{E^\circ/K}^\meta|_{\Gamma_K} : \Gamma_K\rightarrow\Out(\pi_1^\et(E^\circ_{\ol{K}})^\meta)$ carries no more information that the standard abelian Galois representation $\rho_{E/K}|_{\Gamma_K} : \Gamma_K\rightarrow\Aut(\pi_1^\et(E_{\ol{K}})) \cong \GL_2(\Zhat)$, generalizing a result of Davis (c.f. \cite{Davis13}, Theorem 3.3), whose proof in her case relied on showing that all IA-automorphisms were inner.

\sgap

Lastly, to obtain the homogeneous decomposition of $\mM(G)_\QQ$, we prove a rather surprising group-theoretic result:

\begin{thm*}[\ref{thm_IAEnds_are_normal}] Every automorphism in $\IAut(\what{M})$ leaves every open normal subgroup of $\what{M}$ stable.
\end{thm*}

In particular, automorphisms in $\IAut(\what{M})$ descend to every finite quotient. Using this, if $G$ is a finite 2-generated metabelian group, then one obtains that the transitive action of $\Out(\what{M})$ on $\Epi^\ext(\what{M},G)$ can be decomposed into the commuting actions of $\Out^b(\what{M})\cong\GL_2(\Zhat)$ on the source, and $\IAut(\what{M})/\Inn(\what{M})\rightarrow\Out(G)$ on the target. The homogeneous disjoint union decomposition of $\mM(G)_\QQ$ then follows from Galois theory.

\section{Notation and Conventions}
Given a group $N$, and a group $Q$ acting on $N$ on the left, we define the semidirect product $N\rtimes Q$ as the set $N\times Q$ with group operation:
$$(n_1,q_1)\cdot(n_2,q_2) := (n_1\cdot\la{q_1}n_2,q_1q_2)$$
In a group $G$, the commutator of two elements $a,b\in G$ is $[a,b] := aba^{-1}b^{-1}$. The subgroup generated by all such commutators is denoted $G'$.

\sgap

For two groups $A,B$,
$$\Epi^\ext(A,B) := \{\text{Surjective homomorphisms $A\rightarrow B$ modulo conjugation}\}$$
If $p$ is a prime number, let $\ZZ_p$ denote the ring of $p$-adic integers, or equivalently the ring of Witt vectors of $\FF_p$. If $q$ is a prime power, then $\ZZ_q$ denotes the unique unramified extension of $\ZZ_p$ with residue field $\FF_q$, or equivalently the ring of Witt vectors of $\FF_q$.
\sgap

All fundamental groups are by default \'{e}tale fundamental groups.

\sgap

All homomorphisms between profinite groups are supposed continuous - sometimes this will require justification, and generators of profinite groups are topological generators.

\section{Preliminaries on profinite wreath product embeddings}
In \S\ref{ss_general_setting}, we begin by giving profinite analogs of classical wreath product embeddings of a profinite group $G$ (c.f. \cite{Rotman95} Theorem 7.37 and \cite{Rem79} \S2), and defining the notion of a ``universal split extension containing $G$''. The goal of the section is to describe an embedding of the rank 2 free profinite metabelian group $\what{M}$, as an extension of its abelianization $\what{A}$, inside the split extension $\Zhat\ps{\what{A}}^2\rtimes\what{A}$, and show that this embedding is the ``universal split extension containing $G$'' (c.f. Proposition \ref{prop_universal}).

\subsection{The general setting}\label{ss_general_setting}

Let $G$ be a finitely generated profinite group which is an extension of some $Q$ by $N$:
$$1\rightarrow N\rightarrow G\rightarrow Q\rightarrow 1$$
Let $\CF(Q,N)$ denote the group of continuous functions $Q\rightarrow N$, with group operation given by pointwise multiplication in $N$. For any open normal subgroup $U\le N$, let
$$\CF_U = \CF_U(Q,N) := \{f\in\CF(Q,N): f(Q)\subset U\}$$
We make $\CF(Q,N)$ into a topological group by using the subgroups $\CF_U$ as a basis of open neighborhoods of the identity. We have a continuous left $Q$-action on such functions given by
$$(\la{\alpha}f)(q) := f(q\alpha)\qquad \text{for all }q,\alpha\in Q\text{.}$$
If $Q$ is finite, then the topology on $\CF(Q,N)$ is visibly profinite, but if $Q$ is infinite, the open subgroups $CF_U$ are infinite index, and hence the resulting topology is not compact. However, the closed finitely generated $Q$-operator subgroups are still profinite:

\begin{prop}\label{prop_remeslennikov} If $K\le CF(Q,N)$ is finitely generated as a closed $Q$-operator subgroup, then $K$ has a collection $\{K_\alpha\}$ of finite index open normal subgroups each stable under $Q$ with $\bigcap_\alpha K_\alpha = 1$. In particular, $K$ is profinite.
\end{prop}
\begin{proof} See \cite{Rem79} Proposition 10.
\end{proof}

Fix a continuous set-theoretic section $\sigma : Q\rightarrow G$. For $g\in G$, we define the continuous function
$$f_g\in\CF(Q,N)\qquad\text{given by}\qquad f_g(q) := \sigma(q)\cdot g\cdot \sigma(q\ol{g})^{-1}$$
where $\ol{g}$ denotes the image of $g\in G$ in $Q$. Let $X\subset G$ be a finite generating set, and let $\CF_X$ be the closed $Q$-operator subgroup of $\CF(Q,N)$ generated by the set $\{f_x : x\in X\}$. Then $\CF_X$ is a profinite group with continuous $Q$-action, so $\CF_X\rtimes Q$ is profinite.
\begin{thm}[c.f. \cite{Rem79} \S2]\label{thm_profinite_embedding} The map $i_\sigma : G\rightarrow \CF_X\rtimes Q$ given by $i_\sigma(g) = (f_g,\ol{g})$ is a continuous injective homomorphism.
\end{thm}
\begin{proof} We have
\begin{eqnarray*}
f_g(q)\la{\ol{g}}f_h(q) & = & f_g(q)f_h(q\ol{g}) \\
 & = & \sigma(q)\cdot g\cdot\sigma(q\ol{g})^{-1}\cdot\sigma(q\ol{g})\cdot h\cdot\sigma(q\ol{g}\ol{h})^{-1} \\
 & = & \sigma(q)\cdot gh\cdot\sigma(q\ol{gh})^{-1} \\
 & = & f_{gh}(q)
\end{eqnarray*}
The above calculation shows that $i_\sigma$ is a homomorphism, whose kernel is certainly contained in $N$. On the other hand, for $n\in N$, we have $f_n(q) = \sigma(q) n\sigma(q)^{-1}$ which is never the identity $\underline{1_G}\in\CF(Q,N)$ if $n\ne 1_G$, and hence $i_\sigma$ is injective. For continuity, it suffices to check that the ``coordinate maps'' $G\rightarrow Q,\; g\mapsto \ol{g}$ and $G\rightarrow \CF_X,\; g\mapsto f_g$ are continuous. The first is continuous by hypothesis. For the second, for each open normal subgroup $K_\alpha\le\CF_X$ stable under $Q$ as in Proposition \ref{prop_remeslennikov}, $K_\alpha\rtimes Q$ is finite index inside $\CF_X\rtimes Q$, and hence its preimage is finite index, and hence is open since $G$ is finitely generated.
\end{proof}

\begin{remark} This embedding $i_\sigma$ is a profinite analog of the classical wreath product embedding (c.f. \cite{Rotman95} 7.37) of $G$ into $N\wr Q = \left(\prod_{q\in Q}N\right)\rtimes Q$, where we view $\CF_X\subset\CF(Q,N)$ as a subgroup of $\prod_{q\in Q}N$.

\end{remark}




\sgap

For a profinite group $Q$, let $\Zhat\ps{Q}$ be the completed group algebra of $Q$, and let $\PMod(\Zhat\ps{Q})$ be the abelian category of profinite $\Zhat\ps{Q}$-modules, or what is the same, the category of profinite abelian groups with continuous $Q$-action. Let $\PAb$ be the category of profinite abelian groups, and $\EExt(Q,\PAb)$ the category of extensions of $Q$ by objects of $\PAb$. To be precise, the objects of $\EExt(Q,\PAb)$ are profinite groups $\what{G_1}$, equipped with a (continuous) surjection $\pi : \what{G_1}\twoheadrightarrow Q$ with abelian kernel, and morphisms are (continuous) homomorphisms respecting the projection onto $Q$.

\sgap

There is a natural fully faithful functor $\Xi : \PMod(\Zhat\ps{\what{A}})\rightarrow\EExt(Q,\PAb)$ given by $\Xi(N) := N\rtimes Q$. Thus, $\Xi$ sends a profinite abelian $\Zhat\ps{Q}$-module $N$ to the corresponding split extension $N\rtimes Q$.

\begin{defn} Define a functor $\Psi : \EExt(Q,\PAb)\rightarrow\PMod(\Zhat\ps{Q})$ as follows. Given an extension $G$ of $Q$, let $\{j_i : G\rightarrow N_i\rtimes Q\}_{i\in I}$ denote a set of representatives of all equivalence classes of morphisms in $\EExt(Q,\PAb)$ from $G$ to split extensions in $\EExt(Q,\PAb)$, where we say that two morphisms $G\rightarrow N_i\rtimes Q$ and $G\rightarrow N_j\rtimes Q$ are equivalent if there is an isomorphism of extensions $N_i\rtimes Q\rightiso N_j\rtimes Q$ compatible with the maps from $G$. We have the product map
$$\prod_{i\in I} j_i : G\rightarrow\left(\prod_{i\in I}N_i\right)\rtimes Q$$
Let $N_G^\univ\rtimes Q$ be the closed $Q$-operator subgroup of $\left(\prod_{i}N_i\right)\rtimes Q$ generated by the image of $G$ via $\prod_i j_i$, and define $\Psi(G) := N_{G}^\univ$. Thus, $\prod_i j_i$ gives us a natural map
$$\eta_{G} : G\longrightarrow \Psi(G)\rtimes Q = N_{G}^\univ\rtimes Q$$
\end{defn}

If $G\rightarrow Q$ is an epimorphism of finitely generated profinite groups, then by \ref{thm_profinite_embedding}, the embedding $i_\sigma : G\rightarrow \CF_X\rtimes Q$ appears as one of the $j_i$'s, and hence $\eta_G : G\rightarrow N_G^\univ\rtimes Q$ is injective. Thus, $\Psi(G)\rtimes Q$ is the ``universal split extension of $Q$ containing $G$'', and $\Psi(G)$ is the kernel of the extension. It follows from construction that $\Psi$ is the left adjoint of $\Xi$, with $\eta$ as the unit of the adjunction.

\begin{prop}\label{prop_profinite_adjunction} For every $G_1\in\EExt(Q,\PAb)$ and $N_2\in\PMod(\Zhat\ps{Q})$, there is a bijection
$$\Hom_{\Zhat\ps{Q}}(\Psi(G_1),N_2)\rightiso \Hom_Q(G_1,\Xi(N_2))$$
natural in both $G_1$ and $N_2$. In other words, $\Psi$ is the left adjoint of $\Xi$.

\sgap

In particular, the natural map $\eta_{G} : G\rightarrow \Psi(G)\rtimes Q = N_{G}^\univ\rtimes Q$ is characterized by the following universal property. For any $\Zhat\ps{Q}$-module $N$ and any morphism of extensions $h : G\rightarrow N\rtimes Q$, there is a unique $\Zhat\ps{Q}$-linear map $h_* : N_{G}^\univ\rightarrow N$ making the following diagram commute:

\[\begin{tikzcd}
G\arrow[r,"\eta_{G}"]\arrow[rd,"h"'] & N_{G}^\univ\rtimes Q\arrow[d,"\Xi(h_*)"]\\
& N\rtimes Q
\end{tikzcd}\]
\end{prop}

\subsection{The Magnus embedding for the rank 2 free profinite metabelian group $\what{M}$}\label{ss_profinite_magnus}

Let $F := F_2$ be the free group on two generators, with abelianization $A := F/F' \cong \ZZ^2$. Then, $M := F/F''$ is a free metabelian group of rank 2 and is an extension of its abelianization $A$ by its derived subgroup $M' = F'/F''$. Since $M'$ is abelian, $M'$ has the natural structure of a $\ZZ[A]$-module, with $A$ acting by left conjugation:
$$a\cdot m := \tilde{a}m\tilde{a}^{-1} \quad\text{for $m\in M'$ and $a\in A$,}$$
where $\tilde{a}$ is any lift of $a$ to $M$.

\sgap

Let $\what{M}$ be the profinite completion of $M$, or equivalently the maximal pro-metabelian quotient of $\what{F}$, with abelianization $\what{A}$. Then $\what{M}$ is a free profinite metabelian group of rank 2, equipped with an embedding of a discrete dense subgroup $M\subset \what{M}$. Let $\mu_m$ denote the cyclic group of order $m$, written multiplicatively. Then the completed group algebra of $\what{A}$ is
$$\Zhat\ps{\what{A}} := \varprojlim_{n,m}(\ZZ/n)[\mu_m\times\mu_m] = \varprojlim_{n,m}(\ZZ/n)[a_1,a_2]/(a_1^m-1,a_2^m-1)$$

As in the discrete case, $\what{M}'$ admits a continuous left action of $\what{A}$ by conjugation, making $\what{M'}$ into a $\Zhat\ps{\what{A}}$-module. Writing the group operation of $\what{M}$ multiplicatively, we will denote the action of $r\in\Zhat\ps{\what{A}}$ on $m\in\what{M}'$ by exponentiation: $m^r$. Since the abstract group algebra $\Zhat[\what{A}]$ is dense inside $\Zhat\ps{\what{A}}$, this is the unique continuous action of $\Zhat\ps{\what{A}}$ on $\what{M}'$ satisfying:
$$m^{n_1\alpha_1 + n_2\alpha_2 + \cdots + n_k\alpha_k} = \prod_{i=1}^k \alpha_im^{n_i}\alpha_i^{-1}\qquad\text{for $k\in\NN, \; n_i\in\Zhat, \; \alpha_i\in\what{A}$}$$

\sgap

\begin{prop}\label{prop_universal} Let $x_1,x_2$ generate $\what{M}$ with images $a_1,a_2\in \what{A}$. Let $\what{T}$ be the free $\Zhat\ps{\what{A}}$-module on the basis $t_1,t_2$. Then, the embedding $\eta_{\what{M}} : \what{M}\rightarrow \Psi(\what{M})\rtimes\what{A}$ is uniquely isomorphic to the embedding
$$\mu : \what{M}\hookrightarrow \what{T}\rtimes \what{A}\qquad x_i\mapsto (t_i,a_i)$$
and satisfies the universal property of \ref{prop_profinite_adjunction}: For any profinite $\Zhat\ps{\what{A}}$-module $\what{N}$, and morphism $h : \what{M}\rightarrow\what{N}\rtimes\what{A}$, there exists a unique $\Zhat\ps{\what{A}}$-linear map $h_* : \what{T}\rightarrow\what{N}$ making the following diagram commute:
\[\begin{tikzcd}
\what{M}\arrow[r,"\eta_{\what{M}}"]\arrow[rd,"h"'] & \what{T}\rtimes \what{A}\arrow[d,"\Xi(h_*)"]\\
& \what{N}\rtimes\what{A}
\end{tikzcd}\]
\end{prop}
\begin{proof} It suffices to check that the map $\mu$ satisfies the universal property. For any morphism of extensions $\varphi : \what{M}\rightarrow \what{N}\rtimes \what{A}$, let $n_i\in \what{N}$ be defined by the property $\varphi(x_i) = (n_i,a_i)$. This determines a unique map $\what{T}\rightarrow \what{N}$ sending $t_i\mapsto n_i$, making the diagram of the proposition commute, showing that $\mu \cong \eta_{\what{M}}$.
\end{proof}

In the following we will let $\mu := \eta_{\what{M}}$ denote the canonical injection $\mu : \what{M}\hookrightarrow \Psi(\what{M})\rtimes Q$, which makes sense without any choice of generators of $\what{M}$. Sometimes, to be explicit, we will let $\what{T} := \Psi(\what{M})$, then by \ref{prop_universal}, $\Psi(\what{M})$ is a free $\Zhat\ps{\Ahat}$-module of rank 2, and for any choice of generators $x_1,x_2\in\what{M}$ with images $a_1,a_2\in\Ahat$, $t_1,t_2$ will denote the basis of $\what{T}$ satisfying $\mu(x_i) = (t_i,a_i)$.

\sgap

We may realize $\what{T}\rtimes \what{A}$ as the matrix group (where we think of $\what{A}$ as being written multiplicatively):
$$\what{T}\rtimes \what{A} \cong \ttmatrix{\what{A}}{\what{T}}{0}{1} := \left\{\ttmatrix{a}{t}{0}{1} : a\in \what{A},t\in \what{T}\right\}$$
In this language, $\mu$ is the profinite analog of the \emph{Magnus embedding} described in \cite{Bach65}, \S3.




\section{Some rank 2 pro-metabelian calculus}
\subsection{Basic properties of $\what{M}$}
Recall that we write the elements of $\what{A}$ multiplicatively, and $\what{T}$ additively, so the identity of $\what{T}\rtimes\what{A}$ is $(0,1)$. By the definition of $\what{M}$ as the profinite completion of $M$, we have canonical embeddings $M\subset \what{M}$, $A\subset\what{A}$.

\begin{prop}\label{prop_basic} Let $x_1,x_2$ denote a generating pair of $\what{M}$, corresponding to a basis $t_1,t_2$ of $\what{T}$. Then,
\begin{itemize}
\item[(a)] Let $(t,a)\in\what{T}\rtimes\what{A}$, and let $(t',1)\in\what{T}$. Then we have
$$(t,a)^{-1} = (-a^{-1}t,a^{-1}),\quad\text{and}\quad (t,a)(t',1)(t,a)^{-1} = (t,a)(t',1)(-a^{-1}t,a^{-1}) = (at',1)$$
\item[(b)] $\mu([x_1,x_2]) := \mu(x_1x_2x_1^{-1}x_2^{-1}) = ((1-a_2)t_1 + (a_1-1)t_2,1)$
\item[(c)] For any section $\sigma : A\rightarrow M$, $\what{M}'$ is topologically generated by the set $\{\sigma(a)[x_1,x_2]\sigma(a)^{-1} : a\in A\}$.
\item[(d)] If $x_1,x_2$ are moreover generators of $M$, then $M'$ is a free $\ZZ[A]$-module of rank 1 with basis $[x_1,x_2]$.
\item[(e)] Every $\ZZ[A]$-basis of $M'$ can be written as the commutator of a generating pair.
\end{itemize}
\end{prop}
\begin{proof} (a) and (b) follow from direct calculations. (c) and (d) follow from the fact that $F'$ is the free (discrete) group on the basis $\{w[x_1,x_2]w^{-1} : w\in F\}$, and thus $M'$ is generated by $\{\sigma(a)[x_1,x_2]\sigma(a)^{-1} : a\in A\}$. For (e), again assume that $x_1,x_2$ generate $M$. Note that $\ZZ[A]\cong \ZZ[a_1,a_2,a_1^{-1},a_2^{-1}]$ and has unit group the direct product $\{1,-1\}\times A$. Then, any basis of $M'$ is given by $[x_1,x_2]^u$ for some unit $u\in\ZZ[A]^\times$. Conjugating $x_1,x_2$ by $a\in A$ gives rise to the commutator $[x_1^a,x_2^a] = [x_1,x_2]^a$, and exchanging $x_1$ with $x_2$ gives rise to $[x_2,x_1] = [x_1,x_2]^{-1}$. Thus, $\Aut(M)$ acts transitively on the set of generators for $M'$ - this proves (e).
\end{proof}

For a profinite group $G$, and a commutative profinite ring $R$, we define the \emph{completed group algebra}
$$R\ps{G} := \varprojlim (R/I)[G/U]$$
as $I,U$ range over all open ideals, open normal subgroups of $G$. It's well known that $\ZZ_p\ps{\ZZ_p}\cong\ZZ_p\ps{s}$ (c.f. \cite{Wil98} Theorem 7.3.3). Here we give an analogous result for $\Zhat\ps{\what{A}}$.

\begin{prop}\label{hard_prop} Let $\Cyc_p$ be the set of $f\in\ZZ_p[x]$ which are irreducible factors of $x^m-1$ for some integer $m\ge 1$ coprime to $p$. We have an isomorphism of profinite rings\footnote{A profinite ring is an inverse limit of finite (discrete) rings, equipped with the usual inverse limit topology.}
$$\Zhat\ps{\what{A}}\cong\prod_{p\text{ prime}}\prod_{(f_1,f_2)\in \Cyc_p^2}\Big(\ZZ_p[x_1]/(f_1))\otimes\ZZ_p[x_2]/(f_2)\Big)\ps{s_1,s_2} $$
In particular, $\Zhat\ps{\what{A}}$ is a product of complete regular local rings of the form $\ZZ_{p^r}\ps{s_1,s_2}$ with $(r,p) = 1$, the images of $a_1,a_2$ are given by $u_1(1+s_1),u_2(1+s_2)$ where the $u_i\in\prod_q\ZZ_q^\times$ are, in each coordinate, a root of unity of order coprime to $q$ which generates the extension $\ZZ_q/\ZZ_p$.

\sgap

The images of $a_1-1,a_2-1$ are, in each coordinate, both nonzero and coprime (ie, are not both divisible by a common prime element of $\ZZ_q\ps{s_1,s_2}$). Further, for every prime $p$, there is precisely one coordinate $\ZZ_{p^r}\ps{s_1,s_2}$ in which both $a_1-1,a_2-1$ are prime, in which case their images are precisely $s_1,s_2$. In every other coordinate at least one of $a_1-1,a_2-1$ is a unit.

\end{prop}
\begin{proof} By definition, $\Zhat\ps{\what{A}} = \varprojlim_{n,m}(\ZZ/n)[a_1,a_2]/(a_1^m-1,a_2^m-1)$. For each $n,m$, the prime powers dividing $n$ generate comaximal ideals, and so we have a decomposition
$$(\ZZ/n)[a_1,a_2]/(a_1^m-1,a_2^m-1)\cong \prod_{p^r\mid\mid n}(\ZZ/p^r)[a_1,a_2](a_1^m-1,a_2^m-1)$$
These decompositions carry over to the limit, and so we have
$$\Zhat\ps{\what{A}} \cong \prod_p\ZZ_p\ps{\what{A}}$$
Let $\ZZ' := \prod_{p'\ne p}\ZZ_{p'}$ be the prime-to-$p$ part of $\Zhat$, so that $\Zhat = \ZZ_p\times\ZZ'$. Let $\what{A} = \ZZ_p^2\times(\ZZ')^2$. Let $\cotimes$ denote the completed tensor product over $\ZZ_p$. Then we have
$$\ZZ_p\ps{\what{A}} = \ZZ_p\ps{\ZZ_p}\cotimes\ZZ_p\ps{\ZZ_p}\cotimes\ZZ_p\ps{\ZZ'}\cotimes\ZZ_p\ps{\ZZ'}$$
By \cite{Wil98} Theorem 7.3.3, $\ZZ_p\ps{\ZZ_p}\cong\ZZ_p\ps{s}$, where the isomorphism sends a generator of $\ZZ_p$ to $1+s$. We have
$$\ZZ_p\ps{\ZZ_p}\cotimes\ZZ_p\ps{\ZZ_p} = \ZZ_p\ps{s_1}\cotimes\ZZ_p\ps{s_2} := \varprojlim_n\ZZ_p[s_1]/(s_1^n)\otimes\ZZ_p[s_2]/(s_2^n) = \varprojlim_n\ZZ_p[s_1,s_2]/(s_1^n,s_2^n) = \ZZ_p\ps{s_1,s_2}.$$
Next we consider
$$\ZZ_p\ps{\ZZ'} := \varprojlim_{m\;:\; (m,p)=1}\ZZ_p[x]/(x^m-1)$$
Since $(m,p) = 1$, $x^m-1$ factors into distinct irreducibles over $\FF_p$, and hence by Hensel's lemma we may lift such a factorization to $\ZZ_p[x]$:
$$x^m - 1 = f_1\cdot f_2\cdots f_k\in\ZZ_p[x]$$
with each $f_i$ irreducible. Note that every maximal ideal of $\ZZ_p[x]$ has the form $(p,f(x))$ for some $f\in\ZZ_p[x]$ irreducible mod $p$. This implies that the $f_i$ are comaximal in $\ZZ_p[x]$, so we may write
$$\ZZ_p[x]/(x^m-1)\cong \prod_i\ZZ_p[x]/(f_i)$$
where each factor in the product is the unique connected etale extension of $\ZZ_p$ of degree $\deg(f_i)$. Taking the limit over all $m$ coprime to $p$, we get
$$\ZZ_p\ps{\ZZ'} = \prod_{f\in \Cyc_p}\ZZ_p[x]/(f)$$
Given a prime power $q = p^r$, the number of times $\ZZ_q$ appears in this product is precisely the number $N_{p,r}$ of irreducible factors of degree $r$ in $x^{p^r-1}-1\in\FF_p[x]$, or equivalently the number of ($p$-power) Frobenius orbits of generators of the cyclic group $\FF_q^\times = \FF_{p^r}^\times$. The images of the generator $1\in\ZZ'$ in the $N_{p,r}$ copies of $\ZZ_q$ inside $\ZZ_p\ps{\ZZ'}$ form a complete set of representatives of Frobenius orbits of the primitive $(p^r-1)$-th roots of unity in $\ZZ_q$. In particular, none of the coordinates of the image of $1\in\ZZ'$ vanish.

\sgap

Thus, we've shown that
$$\ZZ_p\ps{\what{A}} = \ZZ_p\ps{s_1,s_2}\cotimes\prod_{f_1\in\Cyc_p}\ZZ_p[x_1]/(f_1(x_1))\cotimes\prod_{f_2\in\Cyc_p}\ZZ_p[x_1]/(f_2(x_2))$$
Since completed tensor products distribute over arbitrary direct products (c.f. \cite{Wil98} Prop 7.7.5), we have
$$\ZZ_p\ps{\what{A}} \cong \prod_{(f_1,f_2)\in\Cyc_p^2}\Big(\ZZ_p[x_1]/(f_1(x_1))\otimes\ZZ_p[x_2]/(f_2(x_2))\Big)\otimes\ZZ_p\ps{s_1,s_2}$$
where we are able to replace the completed tensor product with the usual tensor product since in each case the left side is finite $\ZZ_p$-algebra (c.f. \cite{RZ10} Prop 5.5.3(d)). Using the initial observation that $\Zhat\ps{\what{A}} = \prod_p\ZZ_p\ps{\what{A}}$, we get the desired isomorphism. In particular, we may write $\Zhat\ps{\what{A}}$ as a product of complete regular local rings:
$$\Zhat\ps{\what{A}}\cong \prod_q\ZZ_q\ps{s_1,s_2}$$
where the $q$'s run over prime powers $p^r$ with $(r,p) = 1$, with each $q$ appearing multiple times. Under this isomorphism, the image of the generators $a_1,a_2\in\what{A}$ are $u_1(1+s_1),u_2(1+s_2)$, where $u_1,u_2\in\prod_q\ZZ_q^\times$ are, in each coordinate, the images of $x_1,x_2$, and hence are roots of unity of order coprime to $p$ which generate the extension $\ZZ_q/\ZZ_p$. The image of $a_1-1,a_2-1$ are given by $u_1s_1 + (u_1-1),u_2s_2+(u_2-1)$. Thus, for each $i=1,2$ and each coordinate ``$\ZZ_q\ps{s_1,s_2}$'', we have two possibilities:
\begin{itemize}
\item If the image of $u_i$ in $\ZZ_q^\times$ is equal to 1, then the image of $a_i-1$ is $u_is_i = s_i$. If this happens for both $i = 1,2$ then the images of $a_1-1,a_2-1$ give distinct prime elements in $\ZZ_q\ps{s_1,s_2}$.
\item If the image of $u_i$ in $\ZZ_q^\times$ is not equal to 1, then the image of $a_i-1$ is $u_is_i + (u_i-1)$. Since $u_i$ is a root of unity of order coprime to $p$, $u_i-1$ is a unit, and hence $u_is_i+(u_i-1)$ is a unit (in the coordinate $\ZZ_q\ps{s_1,s_2})$.
\end{itemize}
Since the coordinate of $u_i$ in any $\ZZ_q\ps{s_1,s_2}$ is just the image of $x_i$, the only coordinate in which both $u_1 = u_2 = 1$ is the coordinate where $f_1 = f_2 = x-1$, and hence for each $p$ there is precisely one coordinate $\ZZ_p\ps{s_1,s_2}$ in which the image of $a_1-1,a_2-1$ are both prime, and in this case they are just $s_1,s_2$.

\end{proof}

\begin{remark}
If we're only interested in pro-$p$ metabelian groups, the relevant ring is $\ZZ_p\ps{\ZZ_p^2}$, which is simply isomorphic to $\ZZ_p\ps{s_1,s_2}$.
\end{remark}

\begin{defn}\label{def_distinguished} In the previous proposition, we will call the unique coordinate $\ZZ_p\ps{s_1,s_2}$ in which $a_i-1 = s_i$ for both $i = 1,2$ the \emph{distinguished $p$-coordinate}. The constant term of an element in the distinguished $p$-coordinate will be called the \emph{distinguished $p$-coefficient}. As described in the proof, this is the coordinate corresponding to $f_1 = f_2 = x-1$. Collecting these coordinates for each $p$, we may write
$$\Zhat\ps{\what{A}} = \Zhat\ps{s_1,s_2}\times\prod_q\ZZ_q\ps{s_1,s_2}$$
Given $f\in\Zhat\ps{\what{A}}$, we'll call the first coordinate in the above decomposition the \emph{distinguished coordinate} of $f$ - this is just a power series in $s_1,s_2$ with coefficients in $\Zhat$. The constant term of the distinguished coordinate will be called the \emph{distinguished coefficient}. 

\sgap

Let $\epsilon : \Zhat\ps{\what{A}}\rightarrow\Zhat$ be the map induced by functoriality by sending every $a\in\what{A}$ to $1\in\Zhat$. This is the augmentation map, and its kernel is the augmentation ideal. Let $f : \Zhat\rightarrow\Zhat\ps{\what{A}}$ be the ``structure homomorphism'', then the augmentation satisfies $\epsilon\circ f = \id_{\Zhat}$. Furthermore, it follows from the definitions that given $b\in\Zhat\ps{\what{A}}$, its distinguished coefficient is just $\epsilon(b)$. 

\sgap

Let $\langle s_1,s_2\rangle$ be the ideal of $\Zhat\ps{s_1,s_2}$ generated by $s_1,s_2$. The subset of \emph{special elements} of $\Zhat\ps{\what{A}}$ is
$$\Zhat\ps{\what{A}}' := \epsilon^{-1}(1) = \big(1+\langle s_1,s_2\rangle\big)\times\prod_q\ZZ_q\ps{s_1,s_2}$$
In other words, the special elements of $\Zhat\ps{\what{A}}$ are those whose augmentation is 1. For any element $r\in\Zhat\ps{\Ahat}$ we may uniquely decompose $r$ as a product
$$r = us\qquad u\in\Zhat,\; s\in\Zhat\ps{\Ahat}'$$
We call $u$ the ``scalar part'' of $r$, and $s$ the ``special part'' of $r$. We will call
$$\Zhat\ps{\what{A}}^{\times'} := \Zhat\ps{\what{A}}^\times\cap\Zhat\ps{\what{A}}'$$
the subgroup of \emph{special units}. In particular, we have
$$\Zhat\ps{\what{A}}^\times\cong\Zhat^\times\times\Zhat\ps{\what{A}}^{\times'}$$
\end{defn}

\begin{cor}\label{hard_cor} Let $x_1,x_2$ generate $\what{M}$, then $\what{M}'$ is a free $\Zhat\ps{\what{A}}$-module of rank 1, generated by $[x_1,x_2]$.
\end{cor}
\begin{proof} The $\Zhat\ps{\what{A}}$-module structure was described in \S\ref{ss_general_setting}. The result then follows from Proposition \ref{prop_basic}(b) and the fact that the images of $a_1-1,a_2-1$ in $\Zhat\ps{\what{A}}$ are not zero divisors, and that $\what{M}'$ is topologically generated by conjugates of the commutator $[x_1,x_2]$.
\end{proof}

\begin{cor}\label{cor_Mhat_center_free} The center $Z(\what{M})$ of $\what{M}$ is trivial.
\end{cor}
\begin{proof} If $x\in Z(\what{M})$ with image $a\in\what{A}$, then it must act trivially on $\what{M}'$, which implies, by the freeness of the $\Zhat\ps{\what{A}}$-action, that $a = 1\in\what{A}$, so $x\in\what{M}'$, but again since it is central, we must have $a_ix = x$ for $i = 1,2$. Since $x$ is an element of the free $\Zhat\ps{\what{A}}$-module $\what{M}'$ and $a_i-1$ is not a zero-divisor, this can only happen if $x = 0$.
\end{proof}

The following corollary gives us a nice characterization of the image of $\mu$.
\begin{cor}\label{cor_image_of_mu} Let $\varphi : \what{T}\rightarrow\Zhat\ps{\what{A}}$ be the $\Zhat\ps{\what{A}}$-linear map given by $t_i\mapsto a_i-1$. From this, define
$$D : \what{T}\rtimes \what{A}\longrightarrow\Zhat\ps{\what{A}}\qquad (t,a)\mapsto a-1-\varphi(t)$$
Then, for $(t,a)\in\what{T}\rtimes\what{A}$. The following are equivalent:
\begin{itemize}
\item[(a)] The element $(t,a)$ lies in the image of $\mu : \what{M}\rightarrow\what{T}\rtimes\what{A}$.
\item[(b)] $D(t,a) = 0$.
\item[(c)] Writing $t = b_1t_1 + b_2t_2$, we have $b_1(a_1-1) + b_2(a_2-1) = a-1$.
\end{itemize}
\end{cor}
\begin{proof} It's clear that (b)$\iff$(c). We wish to show that (a)$\iff$(b). One can check that $D$ satisfies
$$D((t,a)(t',a')) = D(t+at',aa') = aa'-1-\varphi(t+at') = \cdots = D(t,a) + aD(t',a')$$
and hence is a continuous crossed homomorphism where $\what{T}\rtimes\what{A}$ acts on $\Zhat\ps{\what{A}}$ through $\what{A}$. Furthermore, we have
$$D(t_i,a_i) = a_i-1-\varphi(t_i) = a_i-1-(a_i-1) = 0\qquad i = 1,2$$
Since $(t_1,a_1),(t_2,a_2)$ generate the dense subgroup $M\le\what{M}$ inside $\what{T}\rtimes\what{A}$, this shows that $D(\mu(\what{M})) = 0$, and hence (a)$\Rightarrow$(b). This direction is identical to the first part of the proof of Theorem 3 in \cite{Rem79}.

\sgap

For the other direction, assume (b). Note that one can always find $m\in\what{M}$ such that $\mu(m) = (t',a^{-1})$ for some $t'\in \what{T}$. Thus, $(t,a)\in\mu(\what{M})$ if and only if $(t+at',1)\in\mu(\what{M})$. Thus, it suffices to show that if $t\in\what{T}$ with $D(t,1) := -\varphi(t) = 0$, then $(t,1)\in\mu(\what{M}')$. To see this, let $t = b_1t_1 + b_2t_2$ where $b_i\in\Zhat\ps{\what{A}}$, then by hypothesis we have
$$b_1(a_1-1) + b_2(a_2-1) = 0,\quad\text{equivalently,}\quad b_1(a_1-1) = b_2(1-a_2)\quad\text{ in }\Zhat\ps{\what{A}}$$
By Proposition \ref{hard_prop}, we know that $a_1-1$ and $1-a_2$ are coprime in $\Zhat\ps{\what{A}}$. Thus, there is a $b\in\Zhat\ps{\what{A}}$ such that
$$b_1 = b(1-a_2)\qquad\text{and}\qquad b_2 = b(a_1-1)$$
As desired, we have
$$t := b_1t_1 + b_2t_2 = b(1-a_2)t_1 + b(a_1-1)t_2 = b\cdot\mu([x_1,x_2])\in\mu(\what{M}')$$
\end{proof}
\begin{remark} The linear map $\varphi : \what{T}\rightarrow\Zhat\ps{\what{A}}$ defined in the proof of \ref{cor_image_of_mu} gives rise to a Koszul complex:
$$0\rightarrow \what{T}\wedge\what{T}\rightarrow\what{T}\rightarrow\Zhat\ps{\what{A}}\rightarrow 0$$
where the first map sends $t_1\wedge t_2\mapsto (1-a_2)t_1 + (a_1-1)t_2\in\what{T}$. The fact we are able to find ``$b$'' as in the proof corresponds to the exactness of this Koszul complex, which is equivalent to $a_1-1,a_2-1$ forming a regular sequence in $\Zhat\ps{\what{A}}$, which follows \emph{a fortiori} from the coprimality of $a_1-1,a_2-1$ in $\Zhat\ps{\what{A}}$.
\end{remark}

\subsection{The profinite Bachmuth embedding $\beta$}\label{ss_bachmuth}
In this section we describe the profinite analog of the Bachmuth embedding first defined in the discrete setting in \cite{Bach65}. We will also identify $\what{M}$ as a subgroup of $\what{T}\rtimes\what{A}$ via $\mu$.

\sgap

Let ``IA'' stand for ``(I)dentity on (A)belianization'', then for a group $G$, define $\IAut(G)$ (resp. $\IAEnd(G)$) to be the group (resp. monoid) of automorphisms (resp. endomorphisms) which act as the identity on $G^\ab$. Similarly, we will use $\IOut(G) := \Ker(\Out(G)\rightarrow\Aut(G^\ab))$. Viewing $G$ as an extension of $G^\ab$, IA-endomorphisms are just endomorphisms of $G$, viewed as an extension of $G^\ab$.

\sgap

In the case of $\what{M}$, by universality of the embedding $\mu : \what{M}\hookrightarrow\what{T}\rtimes\what{A}$ (c.f. \S\ref{ss_profinite_magnus}), any IA-endomorphism of $\what{M}$ extends uniquely to a $\Zhat\ps{\what{A}}$-linear endomorphism of $\what{T}$ such that the induced endomorphism of $\what{T}\rtimes\what{A}$ leaves $\mu(\what{M})$ stable. Thus, we have an continuous injective monoid homomorphism
$$\beta : \IAEnd(\what{M})\hookrightarrow \End(\what{T}) \cong M_2(\Zhat\ps{\what{A}})$$
called the \emph{profinite Bachmuth embedding}. Explicitly, if $x_1,x_2$ generate $\what{M}$, with corresponding basis $t_1,t_2$ of $\what{T}$, then for $\gamma\in\IAEnd(\what{M})$, $\beta(\gamma)(t_i)$ is determined by the formula
$$\gamma(x_i) = \gamma(t_i,a_i) = (\beta(\gamma)(t_i),a_i)$$

\sgap

When there can be no confusion, we will often identify elements $\gamma\in\IAEnd(\what{M})$ with their matrices $\beta(\gamma)$, and will use ``the (Bachmuth) \emph{determinant} of $\gamma$'' to refer to $\det(\beta(\gamma))$. A key result, which will be used frequently, is that $\mu(\what{M}')$ is an 1-dimensional ``eigenspace'' for every element of $\beta(\IAEnd(\what{M}))\subset\End(\what{T})$, with eigenvalues given by the Bachmuth determinant (c.f. Proposition \ref{prop_alternative_def_of_det}).

\sgap


\sgap

\begin{example}\label{ex_conjugation} Let $x_1,x_2$ generate $\what{M}$ with images $a_1,a_2\in\what{A}$ and corresponding basis $t_1,t_2\in\what{T}$. Let $\gamma_i\in\IAut(\what{M})$ denote the inner automorphism $x\mapsto x_ixx_i^{-1}$, $i = 1,2$. We compute their matrices here:
$$\gamma_1(x_1) = x_1,\qquad \gamma_1(x_2) = (t_1,a_1)(t_2,a_2)(-a_1^{-1}t_1,a_1^{-1}) = (t_1+a_1t_2 - a_2t_1,a_2) = ((1-a_2)t_1 + a_1t_2,a_2)$$
$$\gamma_2(x_1) = (t_2,a_2)(t_1,a_1)(-a_2^{-1}t_2,a_2^{-1}) = (t_2+a_2t_1 - a_1t_2,a_1) = (a_2t_1 + (1-a_1)t_2,a_1),\qquad \gamma_2(x_2) = x_2$$
$$\gamma_1 = \ttmatrix{1}{1-a_2}{0}{a_1},\qquad \gamma_2 = \ttmatrix{a_2}{0}{1-a_1}{1}$$
In particular, $\det(\gamma_i) = a_i$, and hence inner automorphisms have determinant in $\what{A}$.
\end{example}

\subsection{IA-Endomorphisms of $\what{M}$}\label{ss_IAEnds}
As usual let $x_1,x_2$ generate $\what{M}$ with images $a_1,a_2\in\what{A}$ and corresponding basis $t_1,t_2\in\what{T}$. For any pair $r = (r_1,r_2)\in \Zhat\ps{\what{A}}^2$, define the endomorphism
$$\gamma_r\in\IAEnd(\what{M}) \quad\text{given by}\quad \gamma_r(x_i) = [x_1,x_2]^{r_i}\cdot x_i\quad i = 1,2$$
where the exponent $\cdot^{r_i}$ denotes the $\Zhat\ps{\what{A}}$-action on $\what{M}'$ (c.f. \S\ref{ss_profinite_magnus}). By corollary \ref{hard_cor}, every element of $\what{M}'$ can be written uniquely as $[x_1,x_2]^r$ for some $r\in\Zhat\ps{\what{A}}$, and hence by freeness of $\what{M}$, we have an explicit parametrization of $\IAEnd(\what{M})$ given by a bijection of sets (which depends on the choice of $x_1,x_2$)
\begin{equation}\label{eq_bijection}
\Zhat\ps{\what{A}}^2\cong_{\Sets} \IAEnd(\what{M}) \qquad r\mapsto \gamma_r
\end{equation}
In the notation of Example \ref{ex_conjugation}, we have $\gamma_1 = \gamma_{(0,1)}$ and $\gamma_2 = \gamma_{(-1,0)}$.

\sgap

We wish to determine the subset of $\Zhat\ps{\what{A}}^2$ which correspond to IA-automorphisms.

\sgap

By the Bachmuth embedding, for every $r = (r_1,r_2)\in\Zhat\ps{\what{A}}^2$, the corresponding endomorphism of $\what{M}$ can be represented as a $2\times 2$ matrix over $\Zhat\ps{\what{A}}$. Computing the matrix is straightforward - we have
$$x_i = (t_i,a_i)\mapsto [x_1,x_2]^{r_i}x_i = (r_i(1-a_2)t_1 + r_i(a_1-1)t_2 + t_i,a_i),\qquad i = 1,2$$
so the corresponding matrix is
\begin{equation}\label{eq_matrix_of_gamma_r}
\gamma_r = \ttmatrix{1+r_1(1-a_2)}{r_2(1-a_2)}{r_1(a_1-1)}{1+r_2(a_1-1)} = \ttmatrix{1}{0}{0}{1} + \ttmatrix{1-a_2}{1-a_2}{a_1-1}{a_1-1}\ttmatrix{r_1}{0}{0}{r_2}
\end{equation}
with determinant
$$\det(\gamma_r) = 1+r_1(1-a_2) + r_2(a_1-1).$$

\begin{example}
Using (\ref{eq_matrix_of_gamma_r}), we may compute the action of $\gamma_r$ on the commutator $[x_1,x_2]$:
$$\gamma_r([x_1,x_2]) = \gamma_r\cvector{1-a_2}{a_1-1} = \cvector{\det(\gamma_r)(1-a_2)}{\det(\gamma_r)(a_1-1)} = \det(\gamma_r)\cvector{1-a_2}{a_1-1} = [x_1,x_2]^{\det(\gamma_r)}$$
In particular, since $[x_1,x_2]$ generates a free $\Zhat\ps{\what{A}}$-module, the equality $\gamma_r([x_1,x_2]) = [x_1,x_2]^{\det(\gamma_r)}$ above determines $\det(\gamma_r)$ and can be used to define the determinant:
\end{example}

\begin{prop}\label{prop_alternative_def_of_det} Given any $\gamma\in\IAEnd(\what{M})$, the action of $\gamma$ on $\what{M}'$, viewed as a free $\Zhat\ps{\what{A}}$-module of rank 1, is given by multiplication by $\det(\gamma)$. By freeness, this property determines $\det(\gamma)$ - that is to say, $\det(\gamma)$ is the ``eigenvalue'' of $\gamma$ on $\what{M}'$.
\end{prop}
\begin{proof} Follows from the above example.
\end{proof}

\begin{remark} By the example above, we find that the bijection (\ref{eq_bijection}) endows $\Zhat\ps{\what{A}}^2$ with a monoid structure given by composition ``$\circ$'', where 
$$r\circ r' = r + \det(\gamma_r)r'$$
relative to which the bijection becomes a monoid homomorphism.
\end{remark}

\begin{example}\label{ex_automorphism} Let $\varphi : \what{T}\rightarrow\Zhat\ps{\what{A}}$ be the $\Zhat\ps{\what{A}}$-linear map given by $t_i\mapsto a_i-1$ (c.f. \ref{cor_image_of_mu}). Then $\beta(\gamma_r)\in\End(\what{T})$ be the induced endomorphism, then from (\ref{eq_matrix_of_gamma_r}), we have
$$\varphi\circ\beta(\gamma_r) = [a_1-1\quad a_2-1]\ttmatrix{1+r_1(1-a_2)}{r_2(1-a_2)}{r_1(a_1-1)}{1+r_2(a_1-1)} = [a_1-1\quad a_2-1] = \varphi$$
Thus, if $\beta(\gamma_r)$ is invertible, we also have $\varphi = \varphi\circ\beta(\gamma_r)^{-1}$. By \ref{cor_image_of_mu}, this implies that the automorphism of $\what{T}\rtimes\what{A}$ induced by $\beta(\gamma)$ restricts to an \emph{automorphism} of $\mu(\what{M})$, and thus $\gamma\in\IAEnd(\what{M})$ is an automorphism of $\what{M}$ if and only if $\beta(\gamma)\in\End(\what{T})$ is an automorphism of $\what{T}$.
\end{example}

\begin{prop}\label{prop_det} An element $r\in\Zhat\ps{\what{A}}^2$ corresponds to an automorphism in $\IAEnd(\what{M})$ if and only if
$$\det(\gamma_r) = 1+r_1(1-a_2) + r_2(a_1-1)\in\Zhat\ps{\what{A}}^{\times'}.$$
Furthermore, the image of the determinant map $\det : \Zhat\ps{\what{A}}^2\rightarrow\Zhat\ps{\what{A}}$ sending $r\mapsto\det(\gamma_r)$ is precisely the set of \emph{special elements} of $\Zhat\ps{\what{A}}'$ (those with augmentation 1).
\end{prop}

\begin{proof} To understand the image of the determinant map, we write $\Zhat\ps{\what{A}}$ as a product $\prod_q\ZZ_q\ps{s_1,s_2}$. In each coordinate, if $a_i-1$ is a unit, then by choosing a suitable value of $r_i$ in that coordinate, we can attain any element of $\ZZ_q\ps{s_1,s_2}$. Thus, the only coordinates in which the determinant map is not surjective are the distinguished $p$-coordinates, where it's plain to see that the image is precisely $1+\langle s_1,s_2\rangle$.

\sgap

Since the Bachmuth embedding $\beta$ is a homomorphism, automorphisms have invertible determinant. The converse follows from \ref{ex_automorphism}.
\end{proof}

\begin{remark} In the proof above, note that there are infinitely many coordinates of $\Zhat\ps{\what{A}}$ in which both $a_1-1,a_2-1$ are units. This implies that $\det : \IAut(\what{M})\rightarrow\Zhat\ps{\what{A}}^\times$ is not injective. However, we will see below (Proposition \ref{prop_IA_is_inner}) that it is injective ``modulo inner automorphisms''.
\end{remark}

\sgap

\begin{prop}\label{prop_IA_is_inner} Let $\gamma\in\IAEnd(\what{M})$, then $\gamma = \gamma_r$ for some $r = (r_1,r_2)\in\Zhat\ps{\what{A}}^2$. Then $\gamma$ is an inner automorphism if and only if
$$\det(\gamma_r) = 1 + r_1(1-a_2) + r_2(a_1-1)\quad\text{lies in}\quad\what{A}$$
In particular, letting $\IOut(\what{M}) := \Ker(\Out(\what{M})\rightarrow\GL(\what{A}))$, the map $\det|_{\IAut(\what{M})}$ induces an isomorphism
$$\ol{\det} : \IOut(\what{M})\rightiso \Zhat\ps{\what{A}}^{\times'}/\what{A}$$
whence an exact sequence
$$1\longrightarrow\left(\Zhat\ps{\what{A}}^{\times'}/\what{A}\right)\stackrel{\ol{\det}^{-1}}{\longrightarrow}\Out(\what{M})\stackrel{\ab_*}{\longrightarrow}\GL_2(\Zhat)\rightarrow 1$$
Here we call $\ol{\det}$ the \emph{reduced determinant}.
\end{prop}

\begin{proof} By Example \ref{ex_conjugation}, inner automorphisms have determinant in $\what{A}$, and all determinants in $\what{A}$ are attainable. Conversely, if $\det(\gamma_r)\in\what{A}$, then $\gamma\in\IAut(\what{M})$, and by composing with suitable inner automorphisms, we may assume $\det(\gamma_r) = 1$, in which case we find:
\begin{eqnarray*}
\det(\gamma_r) = 1+r_1(1-a_2)+r_2(a_1-1) & = & 1 \\
r_1(1-a_2)+r_2(a_1-1) & = & 0 \\
r_2(a_1-1) & = & r_1(a_2-1)
\end{eqnarray*}
Using the fact that the $a_i-1\in\Zhat\ps{\what{A}}$ are coprime, as in the proof of Corollary \ref{cor_image_of_mu}, we may find an $\alpha\in\Zhat\ps{\what{A}}$ such that $r_1 = \alpha(a_1-1)$ and $r_2 = \alpha(a_2-1)$. Thus,
$$\gamma_r(x_1) = [x_1,x_2]^{\alpha(a_1-1)}x_1,\qquad \gamma_r(x_2) = [x_1,x_2]^{\alpha(a_2-1)}x_2$$
Now, recall that conjugation of elements in $\what{T}\subset\what{T}\rtimes\what{A}$ by $(t,a)\in\what{T}\rtimes\what{A}$ amounts to multiplication by $a$. Thus, we have $x_1[x_1,x_2]^\alpha x_1^{-1} = [x_1,x_2]^{\alpha a_1}$, whence $[x_1,x_2]^{-\alpha}x_1[x_1,x_2]^\alpha = [x_1,x_2]^{\alpha(a_1-1)}x_1$, and similarly for $x_2$. Thus, $\gamma_r$ is precisely conjugation by $[x_1,x_2]^\alpha$.

\end{proof}

\subsection{General endomorphisms of $\what{M}$}
\begin{defn}[Generalized determinant]\label{def_determinant} Let $c\in M'$ be a $\ZZ[A]$-basis of $M'$, and hence also a $\Zhat\ps{\Ahat}$-basis of $\what{M}'$. For $\gamma\in\End(\what{M})$, let $\det(\gamma)$ be the unique element $\alpha\in\Zhat\ps{\what{A}}^\times$ such that $\gamma(c) = c^\alpha$ (uniqueness follows from the freeness of the $\Zhat\ps{\what{A}}$-module $\what{M}'$, c.f. Corollary \ref{hard_cor}). We call $\det_c(\gamma)$ the \emph{(generalized) Bachmuth determinant} of $\gamma$ relative to the basis $c$ of $M'$. Naturally, $\det_c$ induces a map
$$\ol{\det_c} : \Out(\what{M})\rightarrow\Zhat\ps{\what{A}}^\times/\what{A}$$
which we call the \emph{reduced (generalized) determinant}.
\end{defn}

\begin{remark} While the Bachmuth determinant $\det : \IAEnd(\what{M})\rightarrow\Zhat\ps{\what{A}}$, is canonical, the generalized determinant depends on the choice of basis $c$ of $M'$. It's main usefulness stems from the fact that it is \emph{continuous}, and that the induced map $\ol{\det_c} : \Out(\what{M})\rightarrow\Zhat\ps{\Ahat}^\times/\Ahat$ does not depend on a choice of basis of $M'$.
\end{remark}

Note that any $\gamma\in\End(\what{M})$, inducing $\gamma^\ab\in\End(\what{A})$, acts by functoriality as a ring endomorphism of $\Zhat\ps{\what{A}}$.

\begin{prop}\label{prop_crossed_homomorphism} The generalized determinant is a continuous \emph{crossed} monoid homomorphism:
$$\det_c : \End(\what{M})\longrightarrow\Zhat\ps{\what{A}}\qquad \det_c(\gamma\circ\gamma') = \det_c(\gamma)\cdot\la{\gamma}\det_c(\gamma')$$
whose restriction to $\IAEnd(\what{M})$ agrees with the Bachmuth determinant, and hence is a homomorphism there.
\end{prop}
\begin{proof} Continuity follows from the continuity and the transitivity of the $\Zhat\ps{\what{A}}$-action on $\what{M}'$.

\sgap

Now suppose $\det_c(\gamma') = \sum_{a\in\what{A}}n_a\cdot a$, where $n_a\in\Zhat$ and all but finitely many $n_a$'s are zero. Thus, $\det_c(\gamma')$ lies in the abstract group algebra $\Zhat[\what{A}]$. Then we have
$$\gamma(\gamma'(c)) = \gamma(c^{\sum_a n_a\cdot a}) = \gamma(\prod_a ac^{n_a}a^{-1}) = \prod_a\gamma(a)\gamma(c)^{n_a}\gamma(a)^{-1} = \gamma(c)^{\sum_a n_a\gamma(a)}$$
$$\cdots = \gamma(c)^{\la{\gamma}\left(\sum_a n_a\cdot a\right)} = c^{\det_c(\gamma)\cdot\la{\gamma}\left(\sum_an_a\cdot a\right)} = c^{\det_c(\gamma)\cdot\la{\gamma}\det_c(\gamma')}$$
Since $\Zhat[\what{A}]$ is dense inside $\Zhat\ps{\what{A}}$, and $\det_c$ is continuous, we see that $\det_c$ is indeed a crossed homomorphism. From example \ref{prop_alternative_def_of_det}, we see that it agrees with the Bachmuth determinant on $\IAEnd(\what{M})$, where again we can verify that it is a homomorphism there since the action of $\IAEnd(\what{M})$ on $\Zhat\ps{\what{A}}$ is trivial.
\end{proof}

\begin{lemma}\label{lemma_diagonal_automorphisms} For any $u\in\Zhat^\times$, and any basis $x_1,x_2$ of $\what{M}$, let $c := [x_1,x_2]$, and let $\gamma_u$ be the automorphism of $\what{M}$ sending $(x_1,x_2)\mapsto (x_1,x_2^u)$, then $\gamma_u$ satisfies $\det(\gamma_u^\ab) = \epsilon(\det_{c}(\gamma_u))$, where $\epsilon : \Zhat\ps{\Ahat}\rightarrow\Zhat$ is the augmentation.
\end{lemma}
\begin{proof} We need to show that $\epsilon(\det_{c}(\gamma_u)) = \det(\gamma_u^\ab)$. Clearly $\det(\gamma_u^\ab) = \det(\spmatrix{1}{0}{0}{u}) = u$, so we need to show $\epsilon(\det_{c}(\gamma_u)) = u$. Consider the map $\Zhat\rightarrow\End(\what{M})$ given by
$$n\quad\mapsto\quad \gamma_n:\left\{\begin{array}{lcl}x_1 & \mapsto & x_1 \\ x_2 & \mapsto & x_2^n\end{array}\right.\qquad\text{for any $n\in\Zhat$}$$
This is a continuous monoid homomorphism, so it suffices to show that for any $n\in\ZZ$, the composition
$$\begin{array}{ccccccl}
\Zhat & \longrightarrow & \End(\what{M}) & \stackrel{\det_{c}}{\longrightarrow} & \Zhat\ps{\what{A}} & \stackrel{\epsilon}{\longrightarrow} & \Zhat \\
n & \mapsto & \gamma_n & \mapsto & \det(\gamma_n) & \mapsto & \epsilon(\det(\gamma_n))
\end{array}$$
is the identity. This follows from a straightforward computation. In any group $A$, given $x,y,z\in A$, we have $[x,yz] = [x,y]y[x,z]y^{-1}$. Thus, in $\what{M}$, we have for any $n\ge 1$,
$$[x_1,x_2^n] = [x_1,x_2][x_1,x_2^{n-1}]^{a_2},\qquad\text{and iterating...}\qquad [x_1,x_2^n] = [x_1,x_2]^{1+a_2+a_2^2+ \cdots+a_2^{n-1}}$$
This shows that $\det_c(\gamma_n) = 1+a_2+a_2^2+\cdots+a_2^{n-1}$, which has augmentation $n$ since $\epsilon(a_2) = 1$.
\end{proof}

\begin{cor} If $\gamma\in\Aut(\what{M})$, then for any choice of generators $x,y$ of $\what{M}$, $\gamma([x,y]) = [x,y]^u$ for some unit $u\in\Zhat\ps{\what{A}}^\times$.
\end{cor}
\begin{proof} Let $\sigma$ be the automorphism sending $x_1,x_2$ to $x,y$, and $c := [x_1,x_2]$. Then, since $\det_c$ is a crossed homomorphism, $\det_c(\gamma),\det_c(\sigma)$ are units, and we have
$$\gamma([x,y]) = (\gamma\circ\sigma)([x_1,x_2]) = [x_1,x_2]^{\det_c(\gamma)\cdot\la{\gamma}\det_c(\sigma)} = [x,y]^{\det_c(\sigma)^{-1}\det_c(\gamma)\la{\gamma}\det_c(\sigma)}$$
but the action of $\gamma$ on $\Zhat\ps{\what{A}}$ sends units to units, so we can take $u = \det_c(\sigma)^{-1}\det_c(\gamma)\la{\gamma}\det_c(\sigma)$.
\end{proof}

\begin{cor}\label{cor_surjective_gdet} For any basis $c\in M'$, the generalized determinant $\det_c : \Aut(\what{M})\rightarrow\Zhat\ps{\what{A}}^\times$ is surjective, and every $\Zhat\ps{\Ahat}$-basis of $\what{M}'$ is given by the commutator of a generating pair of $\what{M}$.
\end{cor}
\begin{proof} The surjectivity follows from \ref{prop_det} and \ref{lemma_diagonal_automorphisms}, and the second statement follows from the surjectivity.
\end{proof}


\begin{example}\label{ex_SL2Z} Recall that $F$ is the free group on the generators $x_1,x_2$, and let $M := F/F''$, both having abelianization $A := F/F' \cong \ZZ^2$. There is a canonical isomorphism $\Out(F)\rightiso\GL(A)$ (c.f. \cite{MKS04} \S3.5) which we will use to identify the two groups. Let $\SAut(F)$ be the preimage of $\SL(A)\cong\SL_2(\ZZ)$ under the map $\Aut(F)\rightarrow\Out(F) = \GL(A)$.  It can be computed that $\SAut(F)$ is generated by the following elements:
$$\gamma_E : \left\{\begin{array}{ll} x_1 \mapsto x_2 \\
x_2\mapsto x_1^{-1}
\end{array}\right.\qquad \gamma_T : \left\{\begin{array}{ll}
x_1\mapsto x_2x_1 \\
x_2\mapsto x_2
\end{array}\right.
$$
We will identify $\gamma_E,\gamma_T$ with their images in $\Aut(\what{M})$ by the natural map $\Aut(F)\rightarrow\Aut(\what{M})$. Note that these are \emph{not} IA-automorphisms. Let $c := [x_1,x_2]\in M'$. In $\what{M}$, we can compute:
$$\gamma_E([x_1,x_2]) = x_1^{-1}[x_1,x_2]x_1 = [x_1,x_2]^{a_1^{-1}}\qquad \gamma_T([x_1,x_2]) = x_2[x_1,x_2]x_2^{-1} = [x_1,x_2]^{a_2}$$
Thus, we have $\det_c(\gamma_E) = a_1^{-1},\det_c(\gamma_T) = a_2$, and since $\det_c$ is continuous and $\what{A}$ is a closed subset of $\Zhat\ps{\what{A}}^\times$, we find that $\det_c(\what{\SAut(F)})\subset\what{A}$.

\sgap

Similarly, identifying $\GL(A) = \Out(F)$, we have a natural map $\SL(A)\subset GL(A)=\Out(F)\rightarrow\Out(\what{M})$, which one sees is an injection since the composition
$$\SL(A)\rightarrow\Out(\what{M})\stackrel{\ab_*}{\rightarrow}\Out(\what{A}) = \GL(\what{A})$$
is just the natural inclusion $\SL(A)\subset\GL(\what{A})$. Thus, viewing $\what{\SL(A)}$ as a subgroup of $\Out(\what{M})$, the above shows that $\ol{\det_c}(\what{\SL(A)}) = \{1\}\subset\Zhat\ps{\what{A}}^\times/\what{A}$, and more generally, $\ol{\det_c}(\what{\GL(A)}) = \{1,-1\}\subset \Zhat\ps{\Ahat}^\times/\Ahat$.
\end{example}

\begin{prop}\label{prop_canonical_determinant} The reduced generalized determinant $\ol{\det_c} : \Out(\what{M})\rightarrow\Zhat\ps{\Ahat}^\times/\Ahat$ is independent of a choice of basis $c$ of $M'$, and hereafter we will sometimes refer to it simply as $\ol{det}$.
\end{prop}
\begin{proof} Any other basis of $M'$ can be written as $\alpha(c)$ for some $\alpha\in\Aut(M)$, with image $\ol{\alpha}\in\Out(M) = \GL(A)$. Then, we have
$$\ol{\det_c}(\gamma\circ\alpha) = \ol{\det_c}(\gamma)\cdot\la{\gamma}\ol{\det_c}(\alpha) = \ol{\det_c}(\gamma)\cdot\ol{\det_c}(\alpha)$$
where the last equality follows from the fact that $-1,1\in\Zhat\ps{\Ahat}$ are fixed by the action of $\gamma$ (c.f. \ref{ex_SL2Z}). In particular, this shows that there is some $a\in\Ahat$ with
$$\gamma(\alpha(c)) = c^{a\det_c(\gamma)\det_c(\alpha)} = (c^{\det_c(\alpha)})^{a\det_c(\gamma)} = \alpha(c)^{a\det_c(\gamma)}$$
which shows that $\ol{\det_c} = \ol{\det_{\alpha(c)}}$, which is what we wanted to prove.
\end{proof}

\begin{remark} More generally, we could have defined the generalized determinant relative to any basis of $\what{M}'$ (not requiring that it be a basis of the subgroup $M'$). In this case, since any basis of $\what{M}'$ generates as a $\ZZ[A]$-module a subgroup isomorphic to $M'$, the proposition should be interpreted as saying that the generalized determinant only depends on the choice of a dense embedding $M\hookrightarrow\what{M}$. In our setting, since we defined $\what{M}$ as the profinite completion of $M$, such an embedding comes with the data of $\what{M}$, and hence we may view $\ol{\det}$ as giving a canonical crossed homomorphism $\ol{\det} : \Out(\what{M})\rightarrow \Zhat\ps{\Ahat}^\times/\Ahat$.
\end{remark}

\begin{cor}\label{good_cor} Two generating pairs $x_1',x_2'$ and $x_1'',x_2''$ of $\what{M}$ which agree in the abelianization $\what{A}$ are simultaneously conjugate in $\what{M}$ if and only if their commutators $[x_1',x_2'],[x_1'',x_2'']$ are conjugate.
\end{cor}
\begin{proof} Since the generating pairs agree in $\what{A}$, the unique automorphism $\gamma\in\Aut(\what{M})$ sending $(x_1',x_2')\mapsto (x_1'',x_2'')$ belongs to $\IAut(\what{M})$. The two pairs are conjugate if and only if $\gamma$ is inner, which by Proposition \ref{prop_IA_is_inner} is true if and only if $\det(\gamma)\in\what{A}$, but by definition $[x_1'',x_2''] = \gamma([x_1',x_2']) = [x_1',x_2']^{\det(\gamma)}$, so $\det(\gamma)\in\what{A}$ if and only if $[x_1',x_2']$ is conjugate to $[x_1'',x_2'']$.
\end{proof}


\section{Main Results}\label{section_main_results}

\subsection{Metabelian groups are congruence}\label{ss_metabelian_congruence}
Recall that $F$ is a (discrete) free group of rank 2 with abelianization $A$, and as usual we use the canonical isomorphism $\Out(F)\rightiso\GL(A)$ to identify $\Out(F) = \GL(A)$ (c.f. \cite{MKS04} \S3.5). By a result of Asada \cite{Asa01}, the natural map $\SL(A)\subset\GL(A)=\Out(F)\hookrightarrow\Out(\what{F})$ induces an \emph{injection} $\what{\SL(A)}\hookrightarrow\Out(\what{F})$. This is called the \emph{congruence subgroup property} for $\Aut(F)$. Composing with the natural map $\Out(\what{F})\rightarrow\Out(\what{M})$, we obtain a representation:
$$\what{\SL(A)}\rightarrow\Out(\what{M})$$

\begin{thm}[Metabelian groups are congruence]\label{thm_metabelian_congruence} The image of the outer representation
$$\what{\SL_2(\ZZ)} \cong \widehat{\SL(A)}\rightarrow\Out(\what{M})$$
is precisely $\SL(\what{A})\cong \SL_2(\Zhat)$. In what follows we will often view $\SL(\Ahat)$ as a subgroup of $\Out(\what{M})$.
\end{thm}
\begin{proof} Since $\what{\SL(A)}$ surjects onto $\SL(\what{A})$, it suffices to show that the representation factors through $\SL(A)$. Let $\SAut(F)$ be the preimage of $\SL(A)$ under the map $\Aut(F)\stackrel{\ab}{\longrightarrow}\Out(F) = \GL(A)$. From \ref{ex_SL2Z}, we see that the image of $\what{\SAut(F)}$ also has determinant in $\what{A}$. In particular, for $n\ge 1$, let $U(n)\le\SAut(F)$ be the preimage of the principal congruence subgroup $\Gamma(n) := \Ker(\SL(A)\rightarrow \SL(A/nA))$ in $\SAut(F)$, and let $U(\infty) := \bigcap_{n\ge 1}\ol{U(n)}\le\what{\SAut(F)}$. Then $U(\infty)$ is the closed subgroup of $\what{\SAut(F)}$ which fits into the commutative diagram
\[\begin{tikzcd}
1\arrow[r] & U(\infty)\arrow[r]\arrow[d] & \what{\SAut(F)}\arrow[r,"\ab_*"]\arrow[d] & \SAut(\what{A})\arrow[r]\arrow[d,equal] & 1\\
1\arrow[r] & \bigcap_{n\ge1}\ol{\Gamma(n)}\arrow[r] & \widehat{\SL(A)}\arrow[r] & \SL(\what{A})\arrow[r] & 1
\end{tikzcd}\]
where the vertical maps are given by modding out by inner automorphisms. Thus, $U(\infty)$ is the preimage of the ``congruence kernel'' $\bigcap_{n\ge 1}\ol{\Gamma(n)}$ of $\SL(\what{A})$. Since the image of $U(\infty)$ in $\Aut(\what{M})$ lies in $\IAut(\what{M})$ and have determinants in $\what{A}$, by Proposition \ref{prop_IA_is_inner}, we find that the image of $U(\infty)$ in $\Aut(\what{M})$ consists only of inner automorphisms. In particular, the outer action of $\what{\SL(A)}$ on $\what{M}$ factors through $\SL(\Ahat)$.
\end{proof}

\begin{cor}\label{cor_metabelian_congruence} For any finite 2-generated metabelian group $G$, the set
$$\Epi^\ext(F,G) := \{\text{Surjective homomorphisms $\varphi : F\rightarrow G$ modulo conjugation}\}$$
admits a left action by $\SL_2(\ZZ)\cong\SL(A)\subset\Out(F)$ acting on the domain, and for any surjection $\varphi : F\rightarrow G$, the stabilizer $\Gamma_\varphi\subset\SL_2(\ZZ)$ of its conjugacy class is a congruence subgroup.
\end{cor}
\begin{proof} Since the set $\Epi^\ext(F,G)$ is finite, this $\SL(\ZZ)$-action induces a (continuous) action of $\what{\SL_2(\ZZ)}$. Since $G$ is finite metabelian, we may identify $\Epi^\ext(F,G) = \Epi^\ext(\what{M},G)$, and hence this action factors through the image of $\what{\SL_2(\ZZ)}$ in $\Out(\what{M})$, which by the theorem is $\SL_2(\Zhat)$. Since $\Epi^\ext(\what{M},G)$ is finite, the stabilizers $\Gamma_\varphi$ are the restrictions to $\SL_2(\ZZ)$ of open subgroups of $\SL_2(\Zhat)$, and thus must be congruence.
\end{proof}

\begin{defn} Let $B$ be any group, then for any finite index characteristic subgroup $K\le B$, let $\Gamma[K] := \Ker(\Aut(B)\rightarrow\Aut(B/K))$. Then we say that $\Aut(B)$ has the \emph{congruence subgroup property} if every finite index subgroup of $\Aut(B)$ contains $\Gamma[K]$ for some $K$.
\end{defn}

Theorem \ref{thm_metabelian_congruence} also gives another proof of the following result, first due to Ben-Ezra \cite{BenEzra16}.

\begin{cor} $\Aut(M)$ does not have the congruence subgroup property.
\end{cor}
\begin{proof} Since $\what{M}$ is center-free (\ref{cor_Mhat_center_free}), by \cite{BER11} Lemma 3.1, it suffices to show the result for $\Out(M)$. It is a classical result that $\Out(M) \cong \GL_2(\ZZ)$ (c.f. \cite{Bach65} Theorem 2). Let $\Gamma\le\SL_2(\ZZ)$ be a noncongruence subgroup (of finite index), then $\Gamma$ contains no congruence subgroups, and hence by Corollary \ref{cor_metabelian_congruence}, does not contain the $\SL_2(\ZZ)$-stabilizers of any $\varphi : \what{M}\rightarrow G$, but when $G$ is a characteristic quotient $M/K$ of $M$, these stabilizers are precisely the subgroups $\Gamma[K]$, which proves the result.
\end{proof}


\subsection{Structure of $\Out(\what{M}$)}

In this section we will give an explicit semi-direct product decomposition of $\Out(\what{M})$. First we need some lemmas.

\begin{lemma}\label{lemma_dc_is_detab} Let $\epsilon : \Zhat\ps{\what{A}}\rightarrow\Zhat$ be the augmentation, and $c\in M'$ a $\ZZ[A]$-basis. Then the following diagram is commutative:
\[\begin{tikzcd}
1\arrow[r] & \IAut(\what{M})\arrow[r]\arrow[d] & \Aut(\what{M})\arrow[r,"\ab_*"]\arrow[d,"\det_c"] &\GL_2(\Zhat)\arrow[r]\arrow[d,"\det"] & 1\\
1\arrow[r] & \Zhat\ps{\what{A}}^{\times'}\arrow[r] & \Zhat\ps{\what{A}}^\times\arrow[r,"\epsilon"] & \Zhat^\times\arrow[r] & 1
\end{tikzcd}\]
where both rows are exact.
\end{lemma}
\begin{proof} The exactness is clear. By Proposition \ref{prop_det}, we know that the left square commutes. It remains to check that the right square commutes. By Proposition \ref{prop_det} and Example \ref{ex_SL2Z}, the right square commutes when restricted to $\IAut(\what{M})\cdot\SAut(F)$, and hence by continuity commutes when restricted to $\Ker(\det\circ\,\ab_*)$. Thus, since $\det\circ\,\ab_*$ and $\epsilon\circ\det$ are both homomorphisms,  it suffices to check that for every $u\in\Zhat^\times$, there exists a $\sigma\in\Aut(\what{M})$ satisfying $\det(\ab_*(\sigma)) = \epsilon(\det(\sigma)) = u$. This follows directly from Lemma \ref{lemma_diagonal_automorphisms}. 
\end{proof}

\begin{defn}\label{def_braid_like_automorphisms} Let $\Out^b(\what{M})$ denote the subset
$$\Out^b(\what{M}) := \{\gamma\in\Out(\what{M}) : \ol{\det}(\gamma) = \det(\gamma^\ab)\}$$
of \emph{braid-like outer automorphisms}. Here, the equality of determinants takes place in $\Zhat\ps{\Ahat}^\times/\Ahat$. By \ref{prop_canonical_determinant}, the generalized determinant $\ol{\det}$, and hence $\Out^b(\what{M})$, is independent of a choice of basis of $M'$.
\end{defn}

\begin{prop}\label{prop_braid_likes_give_splitting} The set $\Out^b(\what{M})$ is a subgroup of $\Out(\what{M})$ containing $\SL(\Ahat)$ (c.f. \ref{thm_metabelian_congruence}), and abelianization $\ab_* : \Out(\what{M})\rightarrow\GL(\Ahat)$ sends $\Out^b(\what{M})$ isomorphically onto $\GL(\Ahat)$.
\end{prop}
\begin{proof} For $\gamma,\gamma'\in\Out^b(\what{M})$, $\ol{\det}(\gamma')\in\Zhat^\times$, so we have $\la{\gamma}\ol{\det}(\gamma') = \ol{\det}(\gamma')$, but this implies
$$\ol{\det}(\gamma\circ\gamma') = \ol{\det}(\gamma)\la{\gamma}\ol{\det(\gamma')} = \det(\gamma^\ab)\cdot\la{\gamma}\det(\gamma'^\ab) = \det(\gamma^\ab)\det(\gamma'^\ab) = \det(\gamma^\ab\circ\gamma'^\ab)$$
so $\Out^b(\what{M})$ is indeed a subgroup. It contains $\SL(\Ahat)$ since $\ol{\det}(\SL(A)) = \{1\}\subset\Zhat\ps{\Ahat}^\times/\Ahat$ (c.f. \ref{ex_SL2Z}).

\sgap

For any $\gamma\in\Out(\what{M})$, by \ref{lemma_dc_is_detab} we know that $u_\gamma := \det(\gamma^\ab)\ol{\det}(\gamma)^{-1}$ is an element of $\Zhat\ps{\Ahat}^{\times'}/\what{A}$. By \ref{prop_IA_is_inner}, there is a unique $\sigma\in\IOut(\what{M})$ such that $\ol{\det}(\sigma) = u_\gamma$. Then,
$$\ol{\det}(\sigma\circ\gamma) = \ol{\det}(\sigma)\la{\sigma}\ol{\det}(\gamma) = u_\gamma\ol{\det}(\gamma) = \det(\gamma^\ab) = \det((\sigma\circ\gamma)^\ab)$$
so $\sigma\circ\gamma$ is the unique element of $\ab_*^{-1}(\gamma^\ab)$ which lies in $\Out^b(\what{M})$. This proves that $\Out^b(\what{M})$ is (the image of) a section of the map $\Out(\what{M})\stackrel{\ab_*}{\rightarrow}\GL(\Ahat)$, as desired.
\end{proof}

As a consequence of the discussion above, we obtain an explicit description of $\Out(\what{M})$.

\begin{thm}[Structure of $\Out(\what{M}$)]\label{thm_split} The map $\GL(\Ahat)\rightiso\Out^b(\what{M})\subset\Out(\what{M})$ splits the exact sequence:
$$1\longrightarrow\left(\Zhat\ps{\what{A}}^{\times'}/\what{A}\right)\stackrel{\ol{\det}^{-1}}{\longrightarrow}\Out(\what{M})\stackrel{\ab_*}{\longrightarrow}\GL(\Ahat)\rightarrow 1$$
(c.f. Proposition \ref{prop_IA_is_inner} for the definition of $\ol{\det}^{-1}$). In particular, we have
$$\Out(\what{M}) \cong \left(\Zhat\ps{\what{A}}^{\times'}/\what{A}\right)\rtimes\GL(\Ahat)$$
where the semidirect product action is induced by the natural action of $\GL(\Ahat)$ on $\what{A}$, and in this presentation, the subgroup $1\rtimes\GL(\Ahat)$ corresponds to $\Out^b(\what{M})$. 

\end{thm}
\begin{proof} The splitting follows from \ref{prop_braid_likes_give_splitting}. 
\end{proof}

\begin{cor}\label{cor_unique_decomposition} Every $\gamma\in\Out(\what{M})$ admits a unique decomposition as the composition of an IA-outer automorphism and a braid-like outer automorphism:
$$\gamma = \gamma_{\IA}\circ\gamma_b\qquad \gamma_{\IA}\in\IOut(\what{M}), \gamma_b\in\Out^b(\what{M})$$
If we choose a section $s$ of the determinant map $\Out^b(\what{M})\rightiso\GL(\what{A})\stackrel{\det}{\rightarrow}\Zhat^\times$, we may further decompose $\gamma$ as
$$\gamma = \gamma_{\IA}\circ\gamma_{\SL}\circ \gamma_s \qquad \gamma_s := s(\det(\gamma^\ab)), \;\gamma_{\SL}\in\SL(\Ahat)$$
Moreover, in this case we have
$$\ol{\det}(\gamma) = \ol{\det}(\gamma_{\IA})\det(\gamma^\ab)\in\Zhat\ps{\Ahat}^\times/\Ahat$$
where $\ol{\det}(\gamma_{\IA})\det(\gamma^\ab)$ is the unique decomposition of $\ol{\det}(\gamma)$ into the product of a ``special'' unit (mod $\Ahat$) and a ``scalar'' unit (c.f. \ref{def_distinguished}). Thus, by \ref{prop_IA_is_inner}, $\gamma_{\IA}$ is determined by the ``special part'' of $\ol{\det}(\gamma)$.
\end{cor}
\begin{proof} Follows from the above.
\end{proof}


\begin{cor} The reduced determinant $\ol{\det} : \Out(\what{M})\rightarrow\Zhat\ps{\what{A}}^\times/\what{A}$ is surjective with kernel precisely the image $\SL(\what{A})$ of $\what{\SL(A)}$ (c.f. \S\ref{ss_metabelian_congruence}). In particular, the orbits of $\what{\SL(A)}\cong\what{\SL_2(\ZZ)}$ acting on the domain in $\Epi^\ext(\what{M},\what{M}) = \Out(\what{M})$ are classified by the reduced determinant $\ol{\det}$, or equivalently the group $\Zhat\ps{\what{A}}^\times/\what{A}$ of conjugacy classes of commutators of generating pairs of $\what{M}$.
\end{cor}
\begin{remark} Recall that the determinant on $\Out(\what{M})$ is only a crossed homomorphism, and its kernel is only a subgroup, which is not normal.
\end{remark}
\begin{proof} Surjectivity is precisely \ref{cor_surjective_gdet}. By example \ref{ex_SL2Z}, $\SL(\what{A})$ is certainly contained in the kernel. It follows from \ref{thm_split} that it is precisely the kernel.

\sgap

For the second statement, note that $\what{M}$ is a \emph{Hopfian} group (c.f. \cite{RZ10} Proposition 2.5.2), so any surjective endomorphism $\what{M}\twoheadrightarrow\what{M}$ is in fact an automorphism. It follows from the unique decompositions of \ref{cor_unique_decomposition} that right-multiplication by $\SL_2(\Zhat)$ preserves the reduced determinant of elements in $\Out(\what{M})$. Now suppose we have $\varphi,\varphi'\in\Out(\what{M})$ with $\ol{\det}(\varphi) = \ol{\det}(\varphi')$, then in particular the special parts of their determinants are the same, and hence $\varphi_{\IA} = \varphi'_{\IA}$, and also $\varphi_b = \varphi_b'$ (c.f. \ref{cor_unique_decomposition}). But then since $\det(\varphi_b^\ab) = \det(\varphi_b'^\ab)$ we must have $\varphi_b' = \varphi_b\circ\gamma$ for some $\gamma\in\SL(\Ahat)\subset\Out^b(\what{M})$.
\end{proof}

\begin{remark} If $G$ is a finite 2-generated metabelian group, the $\SL_2(\Zhat)$ orbits on $\Epi^\ext(\what{M},G)$ in general are not completely classified by the conjugacy classes of commutators of generating pairs in $G$. However, in all computed examples, they do seem to be classified by conjugacy classes of lifts of the commutator to a Schur cover of $G$.
\end{remark}

\subsection{IA-Endomorphisms of $\what{M}$ are normal}\label{ss_IAEnds_are_normal}
Let $G$ be a finite 2-generated metabelian group. 
Let $\varphi_0 : \what{M}\rightarrow G$ be a surjection, from which we get a commutative diagram of exact sequences:
\begin{equation}\label{eq_level_G}
\begin{tikzcd}
1\arrow[r] & \what{M}'\arrow[r]\arrow[d,twoheadrightarrow,"\varphi_0'"] & \what{M}\arrow[r,"\ab"]\arrow[d,twoheadrightarrow,"\varphi_0"] &\what{A}\arrow[d,twoheadrightarrow,"\varphi_0^\ab"]\arrow[r] & 1\\
1\arrow[r] & G'\arrow[r] & G\arrow[r,"\ab"] & G^\ab\arrow[r] & 1
\end{tikzcd}
\end{equation}

\begin{thm}[IA-Endomorphisms are normal]\label{thm_IAEnds_are_normal} If $\gamma\in\IAEnd(\what{M})$, then $\gamma$ stabilizes $\Ker(\varphi_0)$, and hence descends to an endomorphism of $G$.
\end{thm}

We will need a few preliminary observations. Our first goal is to describe $\Ker(\varphi_0)$. Recall that $\what{M}'$ is a free $\Zhat\ps{\what{A}}$-module of rank 1. The surjection $\varphi_0' : \what{M}'\rightarrow G'$ is then a surjection of cyclic $\Zhat\ps{\what{A}}$-modules. Thus, $\Ker(\varphi_0')$ is a $\Zhat\ps{\what{A}}$-submodule, and since $\what{M}'$ is cyclic, $\Ker(\varphi_0') = I_{\varphi_0}\what{M}'$ for some closed ideal $I_{\varphi_0}\subset\Zhat\ps{\what{A}}$, and $G'$ is a free $\Zhat\ps{\what{A}}/I_{\varphi_0}$-module of rank 1. We will sometimes abuse notation and identify $\Ker(\varphi_0')$ with $I_{\varphi_0}$.

\sgap

We will need the following observations:
\begin{lemma}\label{lemma_commutation_relation} Let $x,y\in\what{M}$ be any two elements. For any $s\in\Zhat\ps{\what{A}}$ and $z\in\what{M}$ with image $a\in\what{A}$, we have the commutation relation
$$z[x,y]^s = [x,y]^{as}z$$
In particular, for any $k\ge 1$, we have $([x,y]^s z)^k = [x,y]^{s(1+a + a^2 + \cdots + a^{k-1})}z^k$.
\end{lemma}
\begin{proof} We can write $[x,y] = [x_1,x_2]^r$ for some $r\in\Zhat\ps{\what{A}}$. In $\what{T}\rtimes\what{A}$ we may write $x = (t,a)$, and $[x,y]^s = [x_1,x_2]^{rs} = rs[x_1,x_2]_{\what{T}}$. Then we have
$$z[x,y]^s = (t,a) (rs[x_1,x_2]_{\what{T}},1) = (t+ars[x_1,x_2]_{\what{T}},a) = (ars[x_1,x_2]_{\what{T}},1)(t,a) = [x_1,x_2]^{ars}z = [x,y]^{as}z$$
We obtain the second result by expanding and repeatedly using the commutation relation.
\end{proof}

\begin{lemma}\label{lemma_deligne_trick} Suppose $x,y\in\what{M}$ are generators with images $a,b\in\what{A}$, and suppose $\varphi_0(x^d) = \varphi_0([x,y]^s)$ for some $s\in\Zhat\ps{\what{A}}$ (ie $\ab(\varphi_0(x))$ has order dividing $d$). Then 
$$1 + a+a^2+\cdots+a^{d-1}\equiv s(1-b) \mod I_{\varphi_0}$$
\end{lemma}
\begin{remark} Note that the statement of the lemma is \emph{not} symmetric in $x$ and $y$.
\end{remark}
\begin{proof} 
Note that $xyx^{-1} = [x,y]y$, so applying the commutation relation \ref{lemma_commutation_relation}, we have
$$x^iyx^{-i} = x^{i-1}[x,y]yx^{-(i-1)} = [x,y]^{a^{i-1}}x^{i-1}yx^{-(i-1)}$$
Iterating, we have
$$x^{d}yx^{-d} = [x,y]^{1 + a+a^2+\cdots+a^{d-1}}y$$
Right-multiplying both sides by $y^{-1}$ and again applying the commutation relation, we have, ``modulo $I_{\varphi_0}$'',
$$[x,y]^{s(1-b)} = [x,y]^{s}y[x,y]^{-s}y^{-1} \equiv x^dyx^{-d}y^{-1} = [x,y]^{1+a+\cdots+a^{d-1}}\qquad\text{in $G'$}$$
Since $x,y$ generate $\what{M}$, $\what{M}'$ a free $\Zhat\ps{\what{A}}$-module generated by $[x,y]$, so $G'$ is a free $\Zhat\ps{\what{A}}/I_{\varphi_0}$-module generated by $\varphi_0([x,y])$, so we must have
$$1+a+a^2+\cdots+a^{d-1}\equiv s(1-b)\mod I_{\varphi_0}$$
as desired. 
\end{proof}

Since $I_{\varphi_0} = \Ker(\varphi_0') = \Ker(\varphi_0)\cap\what{M}'$, $I_{\varphi_0}$ is a (closed) normal subgroup of $\what{M}$. Since IA-endomorphisms act on $\what{M}'$ by multiplication by their (Bachmuth) determinants, and $\Ker(\varphi_0')$ is a $\Zhat\ps{\what{A}}$-submodule, every IA-endomorphism of $\what{M}$ stabilizes $\Ker(\varphi_0')$. In particular, we have

\begin{lemma} Every IA-endomorphism of $\what{M}$ descends to an IA-endomorphism of $\what{M}/I_{\varphi_0}$.
\end{lemma}

\begin{remark} In general, for a closed ideal $I\subset\Zhat\ps{\what{A}}$, the subgroup $I\what{M}'\subset \what{M}'$ is stabilized by an $\gamma\in\Aut(\what{M})$ if and only if $\gamma(I) \subset I$, where here $\gamma$ acts on $\Zhat\ps{\what{A}}$ via functoriality, through its action on $\what{A}$ (c.f. Proposition \ref{prop_crossed_homomorphism}). From this, the lemma is clear, since IA-automorphisms by definition act trivially on $\what{A}$.
\end{remark}

\sgap

Reducing the top row of diagram (\ref{eq_level_G}) modulo $I_{\varphi_0}$, we get

\begin{equation*}
\begin{tikzcd}
1\arrow[r] & \what{M}'/I_{\varphi_0}\arrow[r]\arrow[d,"\cong"] & \what{M}/I_{\varphi_0}\arrow[r,"\ab"]\arrow[d,twoheadrightarrow,"\ol{\varphi_0}"] &\what{A}\arrow[d,twoheadrightarrow,"\varphi_0^\ab"]\arrow[r] & 1\\
1\arrow[r] & G'\arrow[r] & G\arrow[r,"\ab"] & G^\ab\arrow[r] & 1
\end{tikzcd}
\end{equation*}
and hence we have an isomorphism $\Ker(\ol{\varphi_0})\cong\Ker(\varphi_0^\ab)$. Let us choose a pair of generators $x_1,x_2$ of $\what{M}$ such that for suitable choices $s_1,s_2\in\Zhat\ps{\what{A}}/I_{\varphi_0}$, we have
$$\Ker(\ol{\varphi_0}) = \langle [x_1,x_2]^{-s_1}x_1^{d_1}\mod I_{\varphi_0}, \quad[x_1,x_2]^{-s_2}x_2^{d_2}\mod I_{\varphi_0}\rangle,\qquad\text{and}$$
$$\Ker(\varphi_0) = \langle I_{\varphi_0} ,[x_1,x_2]^{-s_1}x_1^{d_1},[x_1,x_2]^{-s_2}x_2^{d_2}\rangle$$
In particular, in $G$, we have $x_i^{d_i} = [x_1,x_2]^{s_i}$. Theorem \ref{thm_IAEnds_are_normal} will follow from the following proposition:
\begin{prop}\label{prop_stability} Let $I\subset\Zhat\ps{\what{A}}$ be a closed ideal, viewed as the normal subgroup $I\what{M}'\lhd \what{M}$ contained in $\what{M}'$. Let $x_1,x_2$ generate $\what{M}$. Let $s_1,s_2\in\Zhat\ps{\what{A}}$, and $d_1,d_2\ge 1$ be such that the subgroup
$$K := \langle I,[x_1,x_2]^{-s_1}x_1^{d_1},[x_1,x_2]^{-s_2}x_2^{d_2}\rangle \le\what{M}$$
has the property that $\ab : \what{M}/I\rightarrow \what{A}$ restricts to an injection on $K/I$. Then given $r = (r_1,r_2)\in\Zhat\ps{\what{A}}^2$, $K$ is stabilized by the IA-endomorphism $\gamma_r$ (c.f. \S\ref{ss_IAEnds}) if and only if 
$$r_2s_1(a_1-1)\in I\qquad\text{and}\qquad r_1s_2(a_2-1)\in I$$
Setting $\gamma_r = \gamma_{(0,1)}$ and $\gamma_{(-1,0)}$, we see that $K$ is normal if and only if $s_1(a_1-1)\in I$ and $s_2(a_2-1)\in I$. In particular, if $K$ is normal, then it is stabilized by all IA-endomorphisms.
\end{prop}
\begin{proof} By definition, we have
$$\gamma_r([x_1,x_2]^{-s_i}x_i^{d_i}) = [x_1,x_2]^{-s_i\det(\gamma_r)}([x_1,x_2]^{r_i}x_i)^{d_i} = [x_1,x_2]^{-s_i\det(\gamma_r) + r_i(1+a_i+\cdots + a_i^{d_i-1})}x_i^{d_i}\qquad \text{for }i = 1,2$$
Since $\gamma_r$ is IA and $K\subset\what{M}$ is sent isomorphically onto its image in $\what{A}$, $K$ is stabilized by $\gamma_r$ if and only if
$$[x_1,x_2]^{-s_i\det(\gamma_r) + r_i(1+a_i+\cdots+a_i^{d_i-1})}x_i^{d_i} = [x_1,x_2]^{-s_i}x_i^{d_i}\quad\text{in $\what{M}/I$ for $i = 1,2$}$$
or equivalently, using \ref{lemma_deligne_trick} and the formula $\det(\gamma_r) = 1+r_1(1-a_2) + r_2(a_1-1)$ (c.f. \ref{prop_det}), if and only if the exponents of $[x_1,x_2]$ agree modulo $I$. That is to say, if and only if
\begin{eqnarray}
-s_1(1 + r_1(1-a_2) + r_2(a_1-1)) + r_1s_1(1-a_2) - (-s_1) = -s_1r_2(a_1-1) & \text{lies in} & I\qquad\text{and} \label{cond1}\\
-s_2(1 + r_1(1-a_2) + r_2(a_1-1)) - r_2s_2(1-a_1) - (-s_2)= -s_2r_1(1-a_2) & \text{lies in} & I \label{cond2}
\end{eqnarray}
Then, $K$ is normal if and only if it is stabilized by $\gamma_{(0,1)}$ (conjugation by $x_1$) and $\gamma_{(-1,0)}$ (conjugation by $x_2$). This translates into the conditions
$$-s_1(a_1-1)\in I\qquad\text{and}\qquad s_2(1-a_2)\in I$$
which implies the two conditions (\ref{cond1}),(\ref{cond2}) since $I$ is an ideal.
\end{proof}

\begin{cor}\label{cor_isomorphic_GL2_orbits} The action of $\Out(G)$ on $\Epi^\ext(\what{M},G)$ permutes transitively the $\Out^b(\what{M})$ $(\cong\GL_2(\Zhat))$-orbits, and hence they are all isomorphic as sets with $\Out^b(\what{M})$-action.
\end{cor}
\begin{proof} Since $\Out(\what{M})$ acts transitively on $\Epi^\ext(\what{M},G)$ (c.f. \cite{RZ10} Theorem 3.5.8), using \ref{cor_unique_decomposition}, any $\varphi\in\Epi^\ext(\what{M},G)$ can be written as $\varphi_0\circ u\circ\gamma$ for some $u\in\IOut(\what{M}),\gamma\in\Out^b(\what{M})$. By theorem \ref{thm_IAEnds_are_normal}, we have a homomorphism:
$$(\varphi_0)_* : \IOut(\what{M})\rightarrow\Out(G) : u\mapsto u_{\varphi_0}\qquad\text{satisfying}\quad \varphi_0\circ u = u_{\varphi_0}\circ\varphi_0$$
so every element of $\Epi^\ext(\what{M},G)$ can be written as
$$u_{\varphi_0}\circ\varphi_0\circ\gamma\qquad\text{for some $u\in\IOut(\what{M}),\gamma\in\Out^b(\what{M})$}$$
Now $\Epi^\ext(\what{M},G)$ has a natural left action by $\IOut(\what{M})$ via $(\varphi_0)_*$, and a right action by $\Out^b(\what{M})$, which commute with each other. On the other hand, the above proves that together they act transitively, which proves the result.
\end{proof}



\subsection{Geometric setup}\label{ss_geometric_setup}
Here we recall the geometric picture (c.f. \S\ref{section_introduction}) which we will use to relate our above group theoretic results to Galois actions on pro-metabelian fundamental groups. Let $K$ be a subfield of $\Qbar$ with absolute Galois group $\Gamma_K := \Gal(\Qbar/K)$. Let $E$ be an elliptic curve over $K$ with origin $O$, and let $\ol{x} : \Spec\ol{K}\rightarrow\mM(1)_K$ be the corresponding geometric point. Let $E^\circ := E - O$ and let $H_1E := H_1(E(\CC),\ZZ)$, then there is a canonical isomorphism $\pi_1(\mM(1)_{\ol{K}},\ol{x})\cong\what{\SL(H_1E)}$ which we will use to identify the two groups. A choice of such an $E$ induces a semi-direct product decomposition
$$\pi_1(\mM(1)_K,\ol{x}) = \pi_1(\mM(1)_{\ol{K}},\ol{x})\rtimes\Gamma_K = \what{\SL(H_1E)}\rtimes\Gamma_K $$
Passing to homology and composing with $H_1(E^\circ(\CC)\hookrightarrow E(\CC))$ yields the abelianization map
$$\underbrace{\pi_1^\tp(E^\circ(\CC))}_{\cong F}\stackrel{}{\rightarrow} H_1(E^\circ(\CC),\ZZ)\rightiso \underbrace{H_1(E(\CC),\ZZ)}_{\cong\ZZ^2}$$
Relative to this map, by Nielsen's theorem (c.f. \S\ref{ss_problem}), any $\gamma\in\pi_1^\tp(\mM(1)(\CC)) = \SL(H_1E)$ lifts uniquely to an element $\tilde{\gamma}\in\Out(\pi_1^\tp(E^\circ(\CC)))$, from which we obtain a homomorphism $\SL(H_1E)\rightarrow\Out(\pi_1^\tp(E^\circ(\CC)))$. Taking profinite completions, we get a homomorphism
$$\rho_{E^\circ/\ol{K}} : \what{\SL(H_1E)}\rightarrow \Out(\pi_1(E^\circ_{\ol{K}}))$$
which is just the restriction of the representation $\rho_{E^\circ/K}$ (c.f. \S\ref{ss_overview}) to the subgroup $\what{\SL(H_1E)}$. The situation is summarized in the following diagram.
\[\begin{tikzcd}
& & & \Out(\pi_1(E^\circ_{\ol{K}}))\arrow[d] \\
\what{\SL_2(\ZZ)}\rtimes\Gamma_K\cong \what{\SL(H_1E)}\rtimes\Gamma_K\arrow[r,equals] & \pi_1(\mM(1)_K,\ol{x})\arrow[rru,bend left = 25,"\rho_{E^\circ/K}"]\arrow[rr,"\rho^\meta_{E^\circ/K}"]\arrow[rrd,bend right=25,"\rho_{E/K}"'] & & \Out(\pi_1(E^\circ_{\ol{K}})^\meta)\arrow[d,"\ab_*"] \\
& & & \Out(\pi_1(E_{\ol{K}}))\arrow[r,equals] & \GL(\what{H_1E})\cong\GL_2(\Zhat)
\end{tikzcd}\]
By choosing a tangential base point, one can always lift the outer action of $\Gamma_K$ on $\pi_1(E^\circ_{\ol{K}})$ to a true action. In general this action is mysterious, but one thing we know is that if $c\in\pi_1(E^\circ_{\ol{K}})$ generates an inertia subgroup, then for any lifting of the outer action and any $\sigma\in\Gamma_K$ we must have $\sigma(c) = a_\sigma c^{\chi(\sigma)} a_\sigma^{-1}$ for some $a_\sigma\in\pi_1(E^\circ_{\ol{K}})$ where $\chi : \Gamma_K\rightarrow\Zhat^\times$ the cyclotomic character. It follows from the topological picture that for any generating pair $x_1,x_2$ of $\pi_1^\tp(E^\circ(\CC))$, $[x_1,x_2]$ generates an inertia subgroup of $\pi_1(E^\circ_{\ol{K}})$, and hence we have
$$\sigma([x_1,x_2]) = a_\sigma [x_1,x_2]^{\chi(\sigma)}a_\sigma^{-1} \qquad\text{for all $\sigma\in\Gamma_K$}$$
On the other hand, by the Weil pairing, if $\sigma^\ab$ denotes the induced action of $\sigma$ on $\pi_1(E^\circ_{\ol{K}})^\ab \cong \pi_1(E_{\ol{K}})$, then viewing $\sigma^\ab\in\GL(\what{H_1E})\cong\GL_2(\Zhat)$, we must have $\det(\sigma^\ab) = \chi(\sigma)$.

\sgap

To apply our group-theoretic results in this geometric setting, we make the following compatible identifications/definitions:
$$F := \pi_1^\tp(E^\circ(\CC)),\qquad M := F/F'', \qquad A := H_1E = H_1(E(\CC),\ZZ) = F^\ab$$
$$\what{F} := \pi_1(E^\circ_{\ol{K}}),\qquad\what{M} := \pi_1(E^\circ_{\ol{K}})^\meta,\qquad \Ahat := \pi_1(E^\circ_{\ol{K}})^\ab = \pi_1(E_{\ol{K}})$$
Thus, we have the notion of the ``reduced generalized determinant'' $\ol{\det} : \Out(\pi_1(E^\circ_{\ol{K}})^\meta)\rightarrow\Zhat\ps{\what{A}}^\times/\what{A}$. The above discussion proves:

\begin{prop}\label{prop_braid_like} For any $\sigma\in\Gamma_K$, we have
$$\ol{\det}(\rho_{E^\circ/K}(\sigma)) = \det(\rho_{E^\circ/K}(\sigma)^\ab) = \chi(\sigma) \qquad\text{in $\Zhat\ps{\Ahat}^\times/\Ahat$}$$
where the second determinant is the usual $\det : \Out(\pi_1(E_{\ol{K}})) = \GL(\Ahat)\rightarrow\Zhat^\times\subset\Zhat\ps{\Ahat}^\times/\Ahat$.
\end{prop}

\subsection{Geometric consequences at infinite level}\label{ss_geometric_consequences_at_infinite_level}


Recall the setup and identifications of the previous section \S\ref{ss_geometric_setup}, where we now set $K = \QQ$, and set $E$ to be an elliptic curve over $\QQ$. Let $G$ be a finite 2-generated group. In \cite{Chen17} \S3, the first author defined the stack $\mM(G)_\QQ$, which is a (generally disconnected) finite \'{e}tale cover of the stack $\mM(1)_\QQ$ of elliptic curves (c.f. \S\ref{ss_motivation}).

\sgap

The geometric fiber of $\mM(G)_{\QQ}\rightarrow\mM(1)_{\QQ}$ above $E_{\Qbar}$ can be identified with $\Epi^\ext(\pi_1(E^\circ_{\Qbar}),G)$, the elements of which are called $G$-structures on $E_{\Qbar}$. The representation $\rho_{E^\circ/\QQ}$ gives an action of $\pi_1(\mM(1)_\QQ)$ on this set by acting on the domain, whose restriction to the subgroup $\what{\SL(H_1E)} = \pi_1(\mM(1)_{\Qbar},\ol{x})$ corresponds to the map
$$\what{\SL(A)}\hookrightarrow\Out(\what{F})$$
described in \S\ref{ss_metabelian_congruence}. Given $\varphi\in\Epi^\ext(\pi_1(E^\circ_{\Qbar}),G)$, the connected component of $\mM(G)_{\Qbar}$ containing the pair $(E,\varphi)$ is a connected finite \'{e}tale cover of $\mM(1)_{\Qbar}$, which under the Galois correspondence corresponds to the $\what{\SL(A)}\cong\what{\SL_2(\ZZ)}$-orbit of $\varphi$, or equivalently to the subgroup $\Gamma_\varphi := \Stab_{\what{\SL(A)}}\varphi\le\what{\SL(A)}$, and over $\CC$, to the modular curve $\hH/\Gamma_\varphi$.

\sgap

If $G$ is metabelian, then we may identify the geometric fiber of $\mM(G)_\QQ\rightarrow\mM(1)_\QQ$ with the set
$$\Epi^\ext(\pi_1(E^\circ_{\Qbar}),G) \cong \Epi^\ext(\pi_1(E^\circ_{\Qbar})^{\meta},G) = \Epi^\ext(\what{M},G)$$
Thus, one may view elements of $\Epi^\ext(\what{M},G)$ as $G$-structures on the elliptic curve $E_{\Qbar}$. Similarly, elements of $\Epi^\ext(\what{M},\what{M})$ should be viewed as $\what{M}$-structures, or ``metabelian level structures of infinite level'' on $E_{\Qbar}$, or equivalently compatible inverse systems of metabelian structures on $E_{\Qbar}$. The set of such $\what{M}$-structures can be identified with the geometric fiber of the pro-finite \'{e}tale Galois covering
$$\mM(\what{M})_\QQ := \varprojlim_{G}\mM(G)_\QQ$$
where $G$ ranges over all finite 2-generated metabelian groups. Its $\widehat{\SL_2(\ZZ)}$-orbits then correspond to ``modular curves of infinite level $\what{M}$''. The abelianization map $\ab : \what{M}\rightarrow\Ahat$ induces a profinite-\'{e}tale Galois covering $\mM(\what{M})_\QQ\rightarrow\mM(\Ahat)_\QQ$, where $\mM(\Ahat)_\QQ$ is defined to be
$$\mM(\Ahat)_\QQ := \varprojlim_{n\ge 1}\mM(\Ahat/n\Ahat)_\QQ$$

\sgap
Under our identifications, \ref{prop_braid_like} tells us that the image of $\Gamma_\QQ$ under $\rho_{E^\circ/\QQ}^\meta$ is contained in the subgroup of braid-like outer automorphisms $\Out^b(\pi_1(E^\circ_{\Qbar})^\meta)$, and by \ref{prop_braid_likes_give_splitting}, the same is true of $\what{\SL(A)} = \pi_1(\mM(1)_{\Qbar},\ol{x})$. The surjectivity of the cyclotomic character $\chi : \Gamma_\QQ\rightarrow \Zhat^\times$ implies:

\begin{cor}\label{cor_M1_monodromy} The image of the monodromy representation
$$\rho_{E^\circ/\QQ}^\meta : \pi_1(\mM(1)_\QQ)\cong\widehat{\SL(H_1E)}\rtimes\Gamma_\QQ\rightarrow \Out(\pi_1(E^\circ_{\Qbar})^\meta)$$
is precisely the subgroup of braid-like outer automorphisms $\Out^b(\pi_1(E^\circ_{\Qbar})^\meta)$. In particular,
$$\mM(\what{M})_\QQ\cong\bigsqcup_u\mM(\Ahat)_\QQ$$
where $u$ ranges over elements of $\Zhat\ps{\what{A}}^{\times'}/\what{A}\cong\IOut(\what{M})$.
\end{cor}
In the next section we will show that an analogous decomposition holds at all finite levels as well (Theorem \ref{thm_metabelian_is_e_congruence}). For elliptic curves over general subfields $K\subset\Qbar$, the following is a direct consequence of \ref{prop_braid_like} and \ref{prop_braid_likes_give_splitting}.
\begin{cor} For any elliptic curve $E$ over a number field $K$ with associated metabelian Galois representation $\rho_{E^\circ/K}^\meta : \Gamma_K\rightarrow\Out(\pi_1(E^\circ_{\Qbar})^\meta)$, the Galois image $\rho_{E^\circ/K}^\meta(\Gamma_K)\le Out(\pi_1(E^\circ_{\Qbar})^\meta)$ is carried isomorphically onto its image in $\Aut(\pi_1(E^\circ_{\Qbar})^\ab) = \Aut(\pi_1(E_{\Qbar}))\cong\GL_2(\Zhat)$ via $\ab_* : \Out(\pi_1(E^\circ_{\Qbar})^\meta)\rightarrow\Aut(\pi_1(E^\circ_{\Qbar})^\ab)$.
\end{cor}
This is an extension of a theorem of Davis (c.f. \cite{Davis13}, Theorem 3.3), where the result was proven in the case where $\what{M}$ is replaced by ``$M_{2,p}(p)$'' (the rank 2 free pro-$p$ metabelian group of commutator exponent $p$). We note that the result of Davis relied on proving that $\IAut(M_{2,p}(p)) = \Inn(M_{2,p}(p))$, which is far from true in $\what{M}$.

\subsection{Geometric consequences at finite level - Structure of $\mM(G)_\QQ$}

In this section, the setup and identifications made in \S\ref{ss_geometric_setup} remain in force, and in addition we will fix an isomorphism $A\cong\ZZ^2$ via which we will identify $A = \ZZ^2$ and consequently also identify $\Ahat = \Zhat^2$, and $\GL(\Ahat) = \GL_2(\Zhat)$. We will also use the map $\Out(\what{M})\rightarrow\GL(\Ahat)$ to identify $\Out^b(\what{M}) = \GL(\Ahat) = \GL_2(\Zhat)$ (c.f. \ref{prop_braid_likes_give_splitting}).
\sgap

For an integer $n\ge 1$, and a subgroup $H\le\GL_2(\ZZ/n)$, we call the set 
$$\Epi^\ext(\pi_1(E_{\Qbar}),(\ZZ/n)^2)/H$$
of $H$-equivalence classes of surjections the set of ``structures of level $H$'' on $E_{\Qbar}$. It can be identified with the geometric fiber above $E_{\Qbar}$ of the covering $\mM((\ZZ/n)^2)_\QQ/H\rightarrow\mM(1)_\QQ$  (c.f. \cite{Chen17} Theorem 3.6.3).
\begin{thm}\label{thm_metabelian_is_e_congruence} Let $G$ be a finite 2-generated metabelian group. The group $\Aut(G)$ acts (via its quotient $\Out(G)$) as automorphisms of the cover $\mM(G)_\QQ\rightarrow\mM(1)_\QQ$, and we have
\begin{itemize}
\item[(1)] $\Out(G)$ permutes transitively the connected components of $\mM(G)_\QQ$, which are hence all isomorphic, and
\item[(2)] If $G$ is of exponent $e$, then there is a subgroup $H\le\GL_2(\ZZ/e)$ such that each connected component of $\mM(G)_\QQ$ is isomorphic to $\mM((\ZZ/e)^2)_\QQ/H$.
\end{itemize}
In particular, for any elliptic curve $E$ over a subfield $K\subset\Qbar$, $G$-structures on $E$ are equivalent to congruence structures of level $H$.
\end{thm}
\begin{proof} From our identifications, the geometric fiber of $\mM(G)_\QQ\rightarrow\mM(1)_\QQ$ over $E_{\Qbar}$ is given by $\Epi^\ext(\what{M},G)$, which admits an action of $\Out(G)$ on the target, commuting with the monodromy action of $\pi_1(\mM(1)_\QQ)$ on the source. By corollary \ref{cor_M1_monodromy}, the action of $\pi_1(\mM(1)_\QQ)$ coincides with the natural action of $\Out^b(\what{M})$. By corollary \ref{cor_isomorphic_GL2_orbits}, we find that $\Out(G)$ acts transitively on the $\Out^b(\what{M})$-orbits, and hence all $\pi_1(\mM(1)_\QQ)$-orbits are isomorphic. By Galois theory, this implies that all connected components of $\mM(G)_\QQ$ are isomorphic, which proves (1).

\sgap

Let $\varphi : \what{M}\rightarrow G$ be a surjection, and let $\what{H}_\varphi$ denote the stabilizer in $\Out^b(\what{M}) = \GL_2(\Zhat)$ of its conjugacy class. Let $G\Gamma(e) := \Ker(\GL_2(\Zhat)\rightarrow\GL_2(\ZZ/e))$. It will suffice to show that $\what{H}_\varphi \supset G\Gamma(e)$. Indeed, if this holds, then we may set $H := \what{H}_\varphi/G\Gamma(e)$, in which case it follows from \ref{cor_M1_monodromy} that the $\pi_1(\mM(1)_\QQ)$-orbit of $\varphi$ is isomorphic to $\Epi^\ext(\pi_1(E_{\Qbar}),(\ZZ/e)^2)/H$ as sets with $\pi_1(\mM(1)_\QQ)$-action. By Galois theory, this would imply that the component of $\mM(G)_\QQ$ corresponding to $\varphi$ is isomorphic to $\mM((\ZZ/e)^2)_\QQ/H$. The containment $\what{H}_\varphi\supset G\Gamma(e)$ is a consequence of the following lemmas.
\end{proof}

\begin{lemma}\label{lemma_awesome} Let $\Gamma\le\SL_2(\ZZ)$ be a congruence subgroup which contains $1 + eX_i$ for $i = 1,2,3$, where
$$X_1 := \ttmatrix{0}{1}{0}{0},\quad X_2 := \ttmatrix{0}{0}{1}{0},\quad X_3 := \ttmatrix{1}{-1}{1}{-1}.$$
Then $\Gamma$ contains the principal congruence subgroup $\Gamma(e)$.
\end{lemma}
\begin{proof} We may view $\Gamma$ as a subgroup of $\SL_2(\Zhat) = \prod_p \SL_2(\ZZ_p)$. Let $\Gamma_p$ denote its image in $\SL_2(\ZZ_p)$ with closure $\ol{\Gamma_p}$. We wish to show that for every $p\nmid e$, we have $\ol{\Gamma_p} = \SL_2(\ZZ_p)$, and for every $p^r\mid\mid e$, we have $\ol{\Gamma_p}\supset\ol{\Gamma(p^r)}$.

\sgap

First we note that for any $n\in\ZZ$, we have $1+nX_i = (1+X_i)^n$. More generally, for any $p$, the closed subgroup of $\SL_2(\ZZ_p)$ generated by $1+X_i$ is isomorphic to the additive group of $\ZZ_p$ as topological groups, and hence for any $a,b\in\ZZ_p$, we have $(1+aX_i)^b = 1+abX_i$ in $\SL_2(\ZZ_p)$.

\sgap

Now suppose $p\nmid e$. Let $e^{-1}$ denote the inverse in $\ZZ_p$. Then we have $(1+eX_i)^{e^{-1}} = 1+X_i\in\ol{\Gamma_p}$, but $\SL_2(\ZZ)$ is generated by the $1+X_i$'s, and being dense in $\SL_2(\ZZ_p)$, this shows that $\ol{\Gamma_p} = \SL_2(\ZZ_p)$ for $p\nmid e$.

\sgap

Now suppose $e = p^ru$ where $p\nmid u$. For ease of notation we will write $\Gamma := \ol{\Gamma_p}$ and $\Gamma(p^k)$ for $\ol{\Gamma(p^k)}$. Viewing $\SL_2(\ZZ_p)$ as an analytic Lie group, it admits a 3-dimensional formal group law $\mf{F}$ defined over $\ZZ$ and convergent on $(p\ZZ_p)^3$. For $k\ge 1$, let $G(p^k)$ be the group on the set $(p^k\ZZ_p)^3$, with group operation given by $\mf{F}$. Then, for $k\ge 1$ we have an isomorphism of groups
\begin{equation}\label{eq_formal_group}
\Gamma(p^k)\cong G(p^k) \qquad \ttmatrix{1}{0}{0}{1} + \ttmatrix{a}{b}{c}{d}\mapsto (a,b,c)\in(p^k\ZZ_p)^3
\end{equation}
Here, we view $(a,b,c)$ as representing the matrix $\spmatrix{1+a}{b}{c}{1+d}$, where $d$ is uniquely determined by the equation
$$(1+a)(1+d)-bc = 1$$
as a power series in $a,b,c$ with zero constant term. From this perspective, the isomorphism (\ref{eq_formal_group}) sends $1+eX_i$ to $eX_i$. Since $\Gamma$ is an open subgroup, $\Gamma\supset \Gamma(p^s)$ for some $s$, where we may assume $s-1 \ge r$. Then, $\Gamma(p^{s-1})/\Gamma(p^s)$ is isomorphic to $G(p^{s-1})/G(p^s)$ which by standard formal group arguments is isomorphic to the additive group $(p^{s-1}\ZZ_p/p^s\ZZ_p)^3\cong \FF_p^3$ on its underlying (quotient) set. Thus, we have an inclusion of $\FF_p$-vector spaces
\begin{equation}\label{eq_inclusion}
\frac{\Gamma\cap\Gamma(p^{s-1})}{\Gamma(p^s)}\subset\frac{\Gamma(p^{s-1})}{\Gamma(p^s)}\quad\Big(\;\;\cong G(p^{s-1})/G(p^s)\cong \FF_p^3\Big).
\end{equation}
Since $1+up^rX_i\in \Gamma$, as before we can write $(1+up^rX_i)^{u^{-1}p^{s-1-r}} = 1+p^{s-1}X_i\in\Gamma$, and hence $p^{s-1}X_i\in G(p^{s-1})$. Since the (coordinates $a,b,c$ of) $p^{s-1}X_i$ span the space $(p^{s-1}\ZZ_p)^3$, we find that the inclusion (\ref{eq_inclusion}) must be an equality - that is to say, $\Gamma\supset\Gamma(p^{s-1})$. By induction, we find that $\Gamma\supset\Gamma(p^r)$, as desired.
\end{proof}

\begin{lemma}\label{lemma_precursor_to_surprise} Let $G$ be a finite 2-generated metabelian group of exponent $e$. Let $\varphi : \what{M}\rightarrow G$ be a surjection. Let $\what{H}_\varphi\le\Out^b(\what{M})$ denote the stabilizer of the conjugacy class of $\varphi$ in $\Out^b(\what{M})$. Let $G\Gamma(e) := \Ker(\GL_2(\Zhat)\rightarrow\GL_2(\ZZ/e))$. Then, under our identification $\GL_2(\Zhat) = \Out^b(\what{M})$, we have $\what{H}_\varphi \supset G\Gamma(e)$.
\end{lemma}
\begin{proof} Viewing $\SL_2(\ZZ)\subset\GL_2(\Zhat) = \Out^b(\what{M})$, let $\Gamma_\varphi := \what{H}_\varphi\cap\SL_2(\ZZ)$. Let $x_1,x_2$ generate $M$, then the automorphism $(x_1,x_2)\mapsto (x_1,x_1^ex_2)$ fixes all surjections $\what{M}\rightarrow G$, and hence its corresponding matrix $\spmatrix{1}{e}{0}{1}$ and all of its $\SL_2(\ZZ)$-conjugates must lie in $\Gamma_\varphi$. In particular, we find
$$\ttmatrix{1}{e}{0}{1},\quad\ttmatrix{1}{0}{e}{1},\quad\ttmatrix{1+e}{-e}{e}{1-e}\in\Gamma_\varphi$$
By theorem \ref{thm_metabelian_congruence}, $\Gamma_\varphi$ is a congruence subgroup, and hence by the above lemma \ref{lemma_awesome}, we find $\Gamma_\varphi\supset\Gamma(e)$. Next, $G\Gamma(e)$ is generated by $\Gamma(e)$ together with the matrices $\spmatrix{1}{0}{0}{u}$ for $u\in\Zhat^\times, u\equiv 1\mod e$. By Lemma \ref{lemma_diagonal_automorphisms}, $G\Gamma(e)\subset\GL_2(\Zhat) = \Out^b(\what{M})$ is generated by $\Gamma(e)$ and the braid-like outer automorphisms represented by $\gamma_u : (x_1,x_2)\mapsto (x_1,x_2^u)$ for $u\in\Zhat^\times,\;u\equiv 1\mod e$. Since $G$ has exponent $e$, it's clear that any such $\gamma_u$ stabilizes $\varphi$, hence $\what{H}_\varphi\supset G\Gamma(e)$. 
\end{proof}

\begin{remark} The preceding lemmas imply that in general, if $G$ is any finite 2-generated group of exponent $e$ and $\varphi : F\rightarrow G$ is a surjection with $\Gamma_\varphi\le\SL_2(\ZZ)$ a congruence subgroup, then $\Gamma_\varphi\supset\Gamma(e)$.
\end{remark}




\bibliography{/Users/wchen/Documents/latex/references}

\begin{bibdiv}
\begin{biblist}

\bib{Asa01}{article}{
      author={Asada, Mamoru},
       title={The faithfulness of the monodromy representations associated with
  certain families of algebraic curves},
        date={2001},
     journal={Journal of Pure and Applied Algebra},
      volume={159},
      number={2},
       pages={123\ndash 147},
}

\bib{Bach65}{article}{
      author={Bachmuth, Seymour},
       title={Automorphisms of free metabelian groups},
        date={1965},
     journal={Transactions of the American Mathematical Society},
      volume={118},
       pages={93\ndash 104},
}

\bib{BenEzra16}{article}{
      author={Ben-Ezra, David El-Chai},
       title={The congruence subgroup problem for the free metabelian group on
  two generators},
        date={2016},
     journal={Groups, Geometry, and Dynamics},
      volume={10},
       pages={583\ndash 599},
}

\bib{BER11}{article}{
      author={Bux, Kai-Uwe},
      author={Ershov, Mikhail},
      author={Rapinchuk, Andrei},
       title={The congruence subgroup property for $\text{Aut}(\text{F}_2)$: A
  group-theoretic proof of {A}sada's theorem},
        date={2009},
     journal={arXiv preprint arXiv:0909.0304},
}

\bib{Chen17}{article}{
      author={Chen, William~Y.},
       title={Moduli interpretations for noncongruence modular curves},
        date={2017},
     journal={Mathematische Annalen},
       pages={p1\ndash 86},
        note={DOI 10.1007/s00208-017-1575-6},
}

\bib{Davis13}{article}{
      author={Davis, Rachel},
       title={Images of metabelian galois representations associated to
  elliptic curves},
        date={2013},
     journal={Women in Numbers 2: Research Directions in Number Theory},
      volume={606},
       pages={29},
}

\bib{MKS04}{book}{
      author={Magnus, Wilhelm},
      author={Karrass, Abraham},
      author={Solitar, Donald},
       title={Combinatorial group theory: Presentations of groups in terms of
  generators and relations},
   publisher={Courier Corporation},
        date={2004},
}

\bib{Rem79}{article}{
      author={Remeslennikov, Vladimir~Nikanorovich},
       title={Embedding theorems for profinite groups},
        date={1979},
     journal={Izvestiya Rossiiskoi Akademii Nauk. Seriya Matematicheskaya},
      volume={43},
      number={2},
       pages={399\ndash 417},
}

\bib{RZ10}{book}{
      author={Ribes, Luis},
      author={Zalesskii, Pavel},
       title={Profinite groups},
   publisher={Springer},
        date={2000},
}

\bib{Rotman95}{book}{
      author={Rotman, Joseph},
       title={An introduction to the theory of groups, 4th ed},
   publisher={Springer Science \& Business Media},
        date={1995},
      volume={148},
}

\bib{Wil98}{book}{
      author={Wilson, John~S},
       title={Profinite groups},
   publisher={Clarendon Press},
        date={1998},
      volume={19},
}

\end{biblist}
\end{bibdiv}

\end{document}